\documentclass[11pt,reqno]{amsart}
\usepackage[margin=1in]{geometry}
\usepackage{amsmath,amssymb,amsthm,graphicx,amsxtra, setspace}
\usepackage[utf8]{inputenc}
\usepackage{mathrsfs}
\usepackage{hyperref}
\usepackage{upgreek}
\usepackage{mathtools}
\usepackage[dvipsnames]{xcolor}
\usepackage[mathcal]{euscript}
\usepackage{amsmath,units}
\usepackage{pgfplots}
\usepackage{tikz}
\usepackage{verbatim}
\allowdisplaybreaks
\usepackage{dsfont}
\usepackage{fontenc}
\usepackage{textcomp}
\usepackage{marvosym}
\usepackage{eurosym}
\usepackage{upgreek}
\usepackage[pagewise]{lineno}

\DeclareMathAlphabet{\mathpzc}{OT1}{pzc}{m}{it}
\DeclareMathOperator*{\esssup}{ess\,sup}

\usepackage[cyr]{aeguill}

\colorlet{darkblue}{blue!50!black}

\hypersetup{
	colorlinks,%
	citecolor=blue,%
	filecolor=red,%
	linkcolor=red,%
	urlcolor=blue,%
	pdfnewwindow=true,%
	pdfstartview={FitH}
}

\newtheorem{theorem}{Theorem}[section]
\newtheorem{lemma}[theorem]{Lemma}
\newtheorem{proposition}[theorem]{Proposition}

\newtheorem{definition}[theorem]{Definition}

\newtheorem{remark}[theorem]{Remark}
\newtheorem{hypothesis}[theorem]{Hypothesis}

\let\originalleft\left
\let\originalright\right
\renewcommand{\left}{\mathopen{}\mathclose\bgroup\originalleft}
\renewcommand{\right}{\aftergroup\egroup\originalright}

\newcommand{\Tr}{\mathop{\mathrm{Tr}}}
\renewcommand{\d}{\/\mathrm{d}\/}

\def\w{\textbf{W}^{\varepsilon}_{{\theta}^{\varepsilon}}}

\def\e{\varepsilon}

\def\L{\mathbb{L}}
\def\A{\mathrm{A}}
\def\I{\mathrm{I}}
\def\F{\mathrm{F}}
\def\C{\mathrm{C}}

\def\J{\mathrm{J}}
\def\B{\mathrm{B}}
\def\D{\mathrm{D}}
\def\y{\boldsymbol{y}}

\def\Y{\boldsymbol{Y}}
\def\E{\mathbb{E}}
\def\X{\boldsymbol{X}}

\def\x{\boldsymbol{x}}

\def\h{\boldsymbol{h}}
\def\z{\boldsymbol{z}}
\def\v{\boldsymbol{v}}
\def\V{\mathbb{v}}
\def\w{\boldsymbol{w}}
\def\W{\mathrm{W}}
\def\G{\mathrm{G}}
\def\Q{\mathrm{Q}}
\def\M{\mathrm{M}}
\def\N{\mathbb{N}}

\def\no{\nonumber}
\def\V{\mathbb{V}}
\def\wi{\widetilde}
\def\Q{\mathrm{Q}}

\def\u{\mathrm{U}}
\def\P{\mathrm{P}}
\def\u{\boldsymbol{u}}
\def\H{\mathbb{H}}

\newcommand{\R}{\mathbb{R}}

\renewcommand{\d}{\/\mathrm{d}\/}

\newcommand{\Addresses}{{% additional braces for segregating \footnotesize
		\footnote{
			\noindent \textsuperscript{1,2}Department of Mathematics, Indian Institute of Technology Roorkee-IIT Roorkee,
			Haridwar Highway, Roorkee, Uttarakhand 247667, INDIA.\par\nopagebreak
			\noindent  \textit{e-mail:} \texttt{Manil T. Mohan: maniltmohan@ma.iitr.ac.in, maniltmohan@gmail.com.}
			
			\textit{e-mail:} \texttt{Sagar Gautam: sagar\_g@ma.iitr.ac.in, sagargautamkm@gmail.com.}
			
			\noindent \textsuperscript{*}Corresponding author.
			
			\textit{Key words:}  Convective Brinkman-Forchheimer equations,  invariant measure, Kolmogorov equation, carr\'e du champs identity, infinite horizon  problems, essential $m-$dissipativity. 
			
			Mathematics Subject Classification (2020): Primary  47D07, 60H15, 47H06; Secondary 60H30, 49L12
			
}}}

\begin{document}
	
	%	\linenumbers
	
	\title[Kolmogorov equations associated with SCBF]{Kolmogorov equations for stochastic convective Brinkman-Forchheimer equations forced by L\'evy Noise and its application to infinite horizon problems
		\Addresses}
	\author[S. Gautam and M. T. Mohan]
	{Sagar Gautam\textsuperscript{1} and Manil T. Mohan\textsuperscript{2*}}

	\maketitle
	
	\begin{abstract}
		This article examines the Kolmogorov equation corresponding to the following stochastic two- and three-dimensional incompressible ($\nabla\cdot\boldsymbol{u}=0$) convective Brinkman-Forchheimer equations, also known as the damped Navier-Stokes equations, driven by L\'evy noise on the torus: 
		\begin{align*}
		\d\boldsymbol{u}+[-\mu\Delta\boldsymbol{u}+(\boldsymbol{u}\cdot\nabla)\boldsymbol{u}+\alpha\boldsymbol{u}+\beta|\boldsymbol{u}|^{r-1}\boldsymbol{u}+\nabla p]\d t
		=\sqrt{\Q}\d\W+\int_{Z}\sigma(t,z)\widetilde{\pi}(\d t,\d z),
		\end{align*}
		where $\mu,\alpha,\beta>0$ are physical constants; $\mathrm{Q}$ is a non-negative, trace-class operator; $\mathrm{W}$ is a cylindrical Wiener process on a Hilbert space; $\sigma$ represents the jump-noise coefficient; $(Z,\mathscr{B}(Z))$ is a measurable space; $\pi$ is a time-homogeneous Poisson random measure; and $\widetilde{\pi}$ denotes its compensator.
		The main contribution of this work is the establishment of the essential $m$-dissipativity of the corresponding Kolmogorov operator, a property that has received limited attention in the existing literature for systems driven by jump-type noise. \emph{Our main innovation is that, in contrast to traditional techniques which crucially depend on exponential moment estimates, we utilize the intrinsic structure of the absorption term $\alpha\boldsymbol{u}+\beta|\boldsymbol{u}|^{r-1}\boldsymbol{u}$ to dispense with these requirements. This allows us to establish the essential $m$-dissipativity of the Kolmogorov operator without the need for exponential moments.}
		In particular, for $r>3$ in two dimensions and for $3< r \leq 5$, as well as for $r=3=$ under the condition $2\beta\mu \geq 1$, in three dimensions, the absorption term supplies enough regularization to eliminate the need for exponential moment estimates.
		We also establish the ``Carr\'e du Champ identity,'' derive perturbation results for the corresponding Kolmogorov operator, and apply the developed framework to an infinite-horizon stochastic optimal control problem, demonstrating the solvability of the associated infinite-dimensional Hamilton-Jacobi-Bellman (integro-differential) equation.
%		 Moreover, we explicitly highlight three central contributions, the derivation of the Carre du Champ identity, the analysis of perturbation results for the associated Kolmogorov operator, and the application of these findings to an infinite horizon optimal control problem.
	\end{abstract}

	\section{Introduction}\label{sec1}\setcounter{equation}{0}
	 The convective Brinkman-Forchheimer (CBF) equations extend the classical Navier-Stokes equations to model fluid flow in saturated porous media by accounting for viscous diffusion (the Brinkman term), Darcy resistance, and inertial drag arising from the Forchheimer correction (\cite{CLF,SGMTM}).
	 The stochastic convective Brinkman-Forchheimer (SCBF) system captures the impact of random perturbations on fluid motion in saturated porous media. To set the stage, we begin by outlining its mathematical formulation.
	Let $(\Omega,\mathscr{F},\mathbb{P})$ be a complete probability space equipped with the filtration $\{\mathscr{F}_t\}_{t\geq0}$ satisfying the usual conditions. In this work, we consider the SCBF equations perturbed by Gaussian and additve jump noise on a $d-$dimensional torus $\mathbb{T}^d$, where $d\in\{2,3\}$. It describes the fluid flow in terms of the velocity vector field $\u(\cdot):[0,T]\times\mathbb{T}^d\times\Omega\to\R^d$, $T<+\infty$, and the pressure $p(\cdot):[0,T]\times\mathbb{T}^d\times\Omega\to\R$, and is given by
	\begin{equation}\label{scbfjump}
		\left\{
		\begin{aligned}
		\d\u(t)+[-\mu\Delta\u(t)&+(\u(t)\cdot\nabla)\u(t)+\alpha\u(t)+\beta|\u(t)|^{r-1}\u(t)+\nabla p(t)]\d t
	\\&=\sqrt{\Q}\d\W(t)+\int_{Z}\sigma(t,z)\wi\pi(\d t,\d z), \ \text{ in } \ \mathbb{T}^d\times(0,T), \\ \nabla\cdot\u&=0, \ \text{ in } \ \mathbb{T}^d\times[0,T], \\
			\u(0)&=\u_0 \ \text{ in } \ \mathbb{T}^d,\\
			\int_{\mathbb{T}^d}p(x,t)\d x&=0, \ \text{ in } \ (0,T),
		\end{aligned}
		\right.
	\end{equation}
  where $\pi$ is a time-homogeneous Poisson random measure, $\widetilde{\pi}$ denotes its compensator,  $(Z,\mathscr{B}(Z))$ is a measurable space and $\sigma:[0,T]\times \mathbb{T}^d \times Z \to\R^d$ is the noise coefficient. The final condition in \eqref{scbfjump} is imposed to ensure the uniqueness of the pressure $p$. The constant $\mu>0$ is the Brinkman coefficient and represents the effective viscosity of the fluid. The parameters $\alpha>0$ and $\beta>0$ denote the Darcy and Forchheimer coefficients, respectively, with $\alpha$ characterizing the permeability of the porous medium and $\beta$ reflecting its porosity. The exponent $r\in[1,\infty)$ is referred to as the absorption exponent, with $r=3$ serving as the critical value.

	\subsection{Literature survey}
	A substantial literature has developed around the stochastic Navier-Stokes equations (SNSE) and related models since the seminal work of Benssoussan and Temam \cite{ABRT}. Many studies have focused on SNSE driven by Gaussian noise, addressing questions of existence and uniqueness of solutions; see, for example, \cite{ZBEM,AD,ffL,FFDG,JLMSS,RMBLR,MRXZ} and the references therein. In particular, by exploiting the monotonicity of the linear and nonlinear operators together with a stochastic variant of the Minty-Browder technique, the authors of \cite{MTM8} established global existence and uniqueness of strong solutions to the SCBF equations on general unbounded domains under multiplicative Gaussian noise in both two and three dimensions. However, Gaussian perturbations often fail to capture the large fluctuations and sudden unpredictable events encountered in many real-world applications, including finance, economics, physics, biology, and chemistry. 
	To address these shortcomings, L\'evy type perturbations has been introduced, as they are capable of representing discontinuities and jump behaviours that frequently occur in real world phenomena. Consequently, SPDEs driven by L\'evy type noise provide a more realistic framework for modelling such phenomena. The study of Navier-Stokes equations and related models driven by random excitations with jumps, have been widely studied in the literature. 
	
	The SNSE driven by L\'evy noise have been extensively studied in recent years.
    The existence and uniqueness of strong solutions for the two-dimensional SNSE were established via the Galerkin approximation method in \cite{ZBEJ} (see also \cite{ZDYX}). In \cite{BPWF1}, the authors proved the existence and uniqueness of pathwise strong solutions for the 2D SNSE with finite L\'evy measure by employing the local monotonicity method combined with a generalized Minty-Browder technique. Later, the global existence and uniqueness of strong solutions for the 2D SNSE under Lipschitz conditions on bounded sets and linear growth assumptions on the jump coefficient were obtained in \cite{ZXJ}.
	For the three dimensional case, the existence and uniqueness of strong solutions for the damped SNSE, together with exponential stability and the existence of invariant measures, were investigated in \cite{HGHL}. The existence and uniqueness of martingale solutions and the corresponding Markov selections for the 3D SNSE were established in \cite{ZDJZ}. Furthermore, \cite{KSSS1} studied the existence and uniqueness of martingale solutions for both 2D and 3D SNSE driven by It\^o-L\'evy noise in bounded and unbounded domains (see also \cite{KSSS2} for results on invariant measures). More recently, the existence of weak martingale solutions for both 2D and 3D SNSE with jump L\'evy processes was proved in \cite{ZBTSEM}. In addition, the solvability of local mild solutions in $\L^p-$space for the 2D and 3D SNSE was investigated in \cite{BPWF}. Moreover, in \cite{MT2}, the author proved existence and uniqueness of strong solutions to the SCBF equations driven by multiplicative L\'evy noise on bounded domains in both two and three dimensions for $r>3$, and for $r=3$ under the condition $2\beta\mu\geq 1$ (see also \cite{MTMJEE} for large deviation principle of SCBF equations driven by L\'evy noise). Further, using fixed-point arguments, the existence and uniqueness of a pathwise mild solution to SCBF equations up to a random time have been established in \cite{MTMLP}.
	
     The works \cite{ASJZ1} and \cite{ASJZ2} studied Hamilton-Jacobi-Bellman integro-PDEs in infinite-dimensional Hilbert spaces associated with stochastic control problems driven by L\'evy noise, developing a viscosity solution framework together with comparison, existence, and dynamic programming principles, and illustrating the theory through controlled SPDE models. 
%	\begin{align*}
%		\u(\cdot)\in\mathrm{L}^2\big(\Omega;
%		\mathrm{L}^{\infty}(0,T;\L^2(\mathcal{O}))\cap
%		\mathrm{L}^2(0,T;\mathbb{H}_0^1(\mathcal{O}))\big)\cap
%		\mathrm{L}^{r+1}(\Omega;\mathrm{L}^{r+1}(0,T;\L^{r+1}(\mathcal{O})))
%	\end{align*}
%	with $\mathbb{P}-$paths in $\D([0,T];\L^2(\mathcal{O}))$, the space of all
%	c\'adl\'ag functions from $[0,T]$ to $\L^2(\mathcal{O})$.
%	In general, the analytical methods developed for SPDEs driven by Gaussian noise are not directly applicable to those influenced by jump-type noise, necessitating new and more refined techniques. For further details, we refer the reader to the previously cited works and the references therein. 
	% In finite-dimensional spaces, Kolmogorov equations are studied well in the literature; for example, see \cite{VIMR,VNMRK} and the references therein. 
	% 
	Kolmogorov equations associated with SPDEs are of significant interest, as they are partial differential equations whose unknown depends on time and on an infinite-dimensional variable. The theory of infinite-dimensional Kolmogorov equations has been extensively developed (cf. \cite{AMAR,SC1,gdp1,gdp7,FFDC,mr13,MRkm, MH1}). A central difficulty lies in understanding the relationship between the infinitesimal generator $\mathcal{N}_2$ of the transition semigroup in $\mathrm{L}^2$-spaces and the abstract Kolmogorov operator $\mathcal{N}_0$. The main challenge is to prove that $\mathcal{N}_2$	is the closure of $\mathcal{N}_0$, a property that remains open in certain cases. In the literature, this is expressed by stating that $\mathcal{N}_0$ is essentially $m-$dissipative. The essential 
$m-$dissipativity of the Kolmogorov operator associated with the SNSE and related models has been thoroughly investigated (see, for instance, \cite{VBGD,VbGdp,WSkm,DYkm,DYkm1}), where exponential moment estimates play a crucial role. However, to the best of our knowledge, the Kolmogorov equation for SNSE driven by L\'evy noise or for related SPDE models in this setting has not been studied in the existing literature. A closely related result appears in \cite{BXY} (see also \cite{LMan} for the space-time white noise case), where the authors investigate the Kolmogorov operator for the one-dimensional stochastic Burgers equation. Their analysis concerns the transition semigroup acting on weighted spaces of continuous functions, and they prove that the infinitesimal generator coincides with the closure of the associated Kolmogorov operator in a suitable topology, though not in an $\mathrm{L}^2$-framework. See \cite{AMAR,Gdpmn,MH1,mr14} for further insights into infinite-dimensional Kolmogorov equations.
	\subsection{Difficulties and approaches}\label{diffapp}
	We now briefly outline the main challenges and difficulties addressed in this work. To that end, consider the following infinite-dimensional Kolmogorov equations associated with the SCBF system \eqref{scbfabs}:
	\begin{align}\label{kolme}
		&\lambda\uppsi(\y)-\frac{1}{2}\Tr\left[\Q\mathcal{D}_{\y}^2\uppsi(\y)\right]+(\mu \mathcal{A}\y+ \alpha\y+ \mathcal{B}(\y)+\beta\mathcal{C}(\y),\mathcal{D}_{\y}\uppsi(\y))
		\nonumber\\&\quad-
			\int_{Z}\left[\uppsi(\y+\G(z))-\uppsi(\y)-(\G(z),\mathcal{D}_{\y}\uppsi(\y))\right]
		\lambda(\d z)=\mathfrak{f}(\y), \ \lambda>0,
	\end{align}
	where $\uppsi(\cdot):\H\to\R$ is the unknown and $\mathfrak{f}(\cdot):\H\to\R$ is  some given function. Furthermore, here, $\D_{\x}$ denotes the derivative with respect to $\x$, and $\Tr$ denotes the trace. It is widely recognized in the literature that the solution $\X(\cdot)=\mathscr{P}\u(\cdot)$ of the system \eqref{scbfabs} (where $\mathscr{P}:\L^2(\mathcal{O})\to\H$ is the Leray projection) is closely related to the solution of the Kolmogorov equation \eqref{kolme}. This relationship is expressed through the following formal identity (cf. \cite{EANK}):
	\begin{align}\label{yxsta}
	 \uppsi(\x)=\int_0^{\infty} e^{-\lambda t} \E[\mathfrak{f}(\X(t,\x))]\d t, \ \x\in\H,
	\end{align}
 However, establishing that \eqref{yxsta} indeed provides a solution to \eqref{kolme} (in a sense that will be made precise later) constitutes one of the major challenges of this work. The main difficulty arises from the need to justify, first of all, the following derivative formulas for sufficiently regular $\uppsi$, obtained through formal calculations:
	\begin{align}
		\mathcal{D}_{\y}\uppsi(\y)\h&=
		\int_0^{\infty} e^{-\lambda t} \E\left[\big(\mathcal{D}_{\y}\mathfrak{f}(\Y(t,\y)),\mathcal{D}_{\y}\Y(t,\y)\h\big)\right]\d t, \label{df1}\\
		\mathcal{D}_{\y}^2\uppsi(\y)(\h,\h)&=\int_0^{\infty} e^{-\lambda t} \E\left[\mathcal{D}_{\y}^2\varphi(\Y(t,\y))
		\big(\mathcal{D}_{\y}\Y(t,\y)\h,\mathcal{D}_{\y}\Y(t,\y)\h\big)\right]\nonumber\\&\qquad+
		\E\left[\big(\mathcal{D}_{\y}\uppsi(\Y(t,\y)),\mathcal{D}_{\y}^2\Y(t,\y)(\h,\h)\big)\right]\d t
		\label{df2},
	\end{align}
	for $t\geq 0$ and $\y,\h\in\H$. Then, we use It\^o's formula to show that $\uppsi$ is a solution of \eqref{kolme}.
	Next, we apply It\^o's formula in order to demonstrate that $\uppsi$ is indeed a solution of \eqref{kolme}.
	This step is far from straightforward, because the identities in \eqref{df1}-\eqref{df2} involve the first- and second-order derivatives $\mathcal{D}_{\y}\Y(t,\y)$ and $\mathcal{D}_{\y}^2\Y(t,\y)$ of the solution to \eqref{scbfabs}, whose integrability is not known in general. Such regularity is guaranteed only when the coefficients $\mathcal{B}(\cdot)$ and $\mathcal{C}(\cdot)$ are sufficiently smooth and the covariance operator $\Q$ is of trace class (see \cite[Theorem 7.5.1, Chapter 7]{gdp7}).

	%	However, this is not straightforward since the formulas \eqref{df1}-\eqref{df2} involve the first and second-order derivatives $\D_{\x}\X(t,\x)$ and $\D_{\x}^2\X(t,\x)$ of the solution to \eqref{scbfabs}, whose integrability is not known in general. 
	%	But this is not easy, as the above derivative formulae involve the derivatives $\D_{\x}\X(t,\x)$ and $\D_{\x}^2\X(t,\x)$ of the solution of \eqref{32}, which are, in general, we do not know whether they are integrable or not. 
%	This is true only when the coefficients $\mathcal{B}(\cdot)$ and $\mathcal{C}(\cdot)$ are regular and $\Q$ is of trace class (see \cite[Theorem 7.5.1, Chapter 7]{gdp7}). 

To address this challenge, we analyze the Kolmogorov equation \eqref{kolme} within the space $\mathrm{L}^2(\H,\eta)$, where $\eta$ represents the invariant measure of the transition semigroup $\{\P_t\}_{t\geq 0}$. A substantial body of literature examines the existence and uniqueness of invariant measures for SPDEs (see \cite{gdp2,AD,MHJC} and references therein). For the SCBF system \eqref{scbfabs}, the existence of an invariant measure is derived directly from energy estimates (see Theorem \ref{extunjump}), in conjunction with the classical Krylov-Bogoliubov theorem. Uniqueness is guaranteed by the exponential stability of the solution $\Y(\cdot)$ to \eqref{scbfabs} (see Theorem \ref{uniqinv}).

%	We point out that unlike other works on Kolmogorov equations for SNSE and related models (see \cite{VBGD,VBGA,VbGdp,gdp1,sgmtm8}), the uniqueness of the invariant measure is heavily depend on the bound of the derivative of the solution $\X(\cdot)$ of the system \eqref{scbfabs}. The main obstacle arises from the absence of an exponential stability estimate for such systems. This issue was resolved in \cite{MT2}, where the author proves the uniqueness of the invariant measure for the SCBF system \eqref{scbfabs} for $r>3$ and $r=3$ with $2\beta\mu\geq1$ in dimensions two and three, without relying on any bounds for the derivative of the solution $\X(\cdot)$.

A central and technically demanding part of our analysis is the proof of the essential $m$-dissipativity of the Kolmogorov operator associated with the SCBF system \eqref{scbfabs}. Specifically, one must verify that the Kolmogorov operator $\mathcal{N}_0$ (see \eqref{4p5}) is the closure of the infinitesimal generator $\mathcal{N}_2$ corresponding to the transition semigroup $\{\P_t\}_{t\ge 0}$. In the classical SPDE framework (and in fluid dynamics models), this step typically relies on delicate exponential moment estimates (see, for example, \cite{VBGD,VBGA,VbGdp,gdp1,DYkm,DYkm1}). For the L\'evy noise case, however, deriving such exponential moments becomes significantly more challenging. Indeed, to the best of our knowledge, exponential moment bounds for the SNSE driven by L\'evy noise are largely absent from the literature. The few available results, such as \cite{MTKSS,MTKSS1}, where exponential estimates are obtained only up to a stopping time using dynamic programming techniques for the 2D SNSE and the stochastic Burgers equation, require highly intricate and technical arguments.

		\subsection{Novelties and advantages}\label{noveadv}
		The absorption term in \eqref{scbfjump} is central to our analysis and is the key structural feature that permits a substantial simplification of the probabilistic estimates required in the study of the SCBF system \eqref{scbfabs} over classical SNSE. In earlier works \cite{VBGD,VBGA,VbGdp,gdp1,WSkm,DYkm,DYkm1} and references therein, the $m$-dissipativity of the Kolmogorov operator and the bounds on the derivatives of the solution $\Y(\cdot)$ to the SCBF system \eqref{scbfabs} were obtain under a substantial technical hypothesis: the existence of an exponential moment estimate $$\E\left[\int_{\H}e^{\kappa\|\y\|_{\H}^2}\d\eta\right]<\infty,$$ for some $\kappa>0$. The exponential moment bound above is then employed to control mixed exponential-interpolation type integrals of the form
	\begin{align}\label{mxdterp}
		\int_{\H}e^{\kappa\|\y\|_{\H}^2}\|\mathcal{A}^{\delta}\y\|_{\H}^{m}
		\|\mathcal{A}^{\delta+\frac12}\y\|_{\H}^{2}\d\eta, \ 
		\text{ for some } \  m\in\N \ \text{ and for appropriate } \ \delta>0.
	\end{align}
	These estimates play a key role in establishing the integrability of $\|\mathcal{A}\y\|_{\H}^2$ with respect to the invariant measure $\eta$ (that is, bounds on $\int_{\H}\|\mathcal{A}\y\|_{\H}^2\d\eta$), particularly in the presence of nonlinear terms; see, for example, \cite{VBGD, VbGdp, sgmtm8}.
	%which in turn are crucial for proving the integrability of $\|\mathcal{A}\x\|_{\H}^2$ with respect to the invariant measure $\eta$ (that is, estimates of $\int_{\H}\|\mathcal{A}\y\|_{\H}^2\d\eta$) due to the presence of nonlinearities, for instance, see the works \cite{VBGD,VbGdp,sgmtm8}. 
%	Due to the presence of the nonlinear operators $\mathcal{B}(\cdot)$ and $\mathcal{C}(\cdot)$ in the SCBF system \eqref{scbfabs}, the integrability of the term $\|\mathcal{A}\x\|_{\H}^2$ with respect to the invariant measure $\eta$ is not immediate and that is why one has to relay on \eqref{mxdterp} (see our previous work \cite{sgmtm8} in this context for 2D SCBF system with $r=1,2,3$). It is mainly required while proving $\mathcal{N}_2$ is an extension of $\mathcal{N}_0$.
	Because the SCBF system \eqref{scbfabs} involves the nonlinear operators  $\mathcal{B}(\cdot)$ and $\mathcal{C}(\cdot)$, the integrability of the quantity $\|\mathcal{A}\y\|_{\H}^2$ with respect to the invariant measure $\eta$ is not immediate. This necessitates the use of estimates of the form \eqref{mxdterp} (see, in particular, our earlier work \cite{sgmtm8} for the 2D SCBF system with $r=1,2,3$). Such bounds are especially important in establishing that $\mathcal{N}_2$ indeed extends $\mathcal{N}_0$. 
	
	Importantly, the dissipation induced by the absorption term enable us to bypass this argument altogether. In the SCBF system \eqref{scbfabs}, the regimes  $r>3$ and $r=3$ with $2\beta\mu\geq1$ (in dimensions $d=2,3$) guarantee that the absorption term contributes strong enough coercivity to yield the required a-priori estimates for the solutions and its derivatives (see Lemma \ref{lem4.4} and Proposition \ref{lem5.2}). This enhanced dissipativity allows us to carry out the entire analysis without relying on exponential moment estimates, in sharp contrast to the approach taken in \cite{VBGD,VBGA,VbGdp,gdp1,DYkm,DYkm1}. Consequently
	\begin{itemize}
		\item we do not need to compute the complicated integral of the form \eqref{mxdterp};
		\item the proof that $\mathcal{N}_2$ is the closure of $\mathcal{N}_0$ becomes significantly simpler because the key estimate $\int_{\H}\|\mathcal{\A}\y\|_{\H}^2\d\eta<\infty$ can be obtained directly from the enhanced dissipation induced by the absorption term (see Lemma \ref{lem4.4} and Proposition \ref{lem5.2}). 
	\end{itemize}
	
	Another major novelty of our work is that we treat the Kolmogorov equation associated with the SCBF system \eqref{scbfabs} driven by additive L\'evy noise. \emph{To the best of our knowledge, this is the first work addressing essential $m$-dissipativity of the Kolmogorov operator associated with the SCBF system \eqref{scbfabs} (and, more generally the related NSE models) in the presence of L\'evy noise.} In the existing SNSE literature (see \cite{VBGD, VBGA, VbGdp, gdp1}), extensions to equations driven by jump noise are seldom considered. This is largely because the standard techniques rely crucially on exponential moment estimates, which are difficult, or in many cases impossible, to derive for general L\'evy noise.  By exploiting the absorption-induced coercivity described above, we bypass this difficulty and obtain the essential $m$-dissipativity of the Kolmogorov operator in the L\'evy noise (additive) setting without relying on exponential moment estimates. Alongside the dissipativity analysis, we also establish the Carr\'e du Champ identity and perturbation results for the Kolmogorov operator associated with the L\'evy driven SCBF system \eqref{scbfabs} (see Propositions \ref{carredu} and \ref{pertubkol}). Since the L\'evy case has not previously been addressed in this framework, we provide full proofs of these identities and estimates.
	
	To summarize, the main advantages and novelties of this paper compared with earlier works \cite{VBGD,VBGA,VbGdp,gdp1,DYkm,DYkm1} are:
	\begin{enumerate}
		\item\textbf{No exponential moments are required.} We do not assume or use exponential moment bounds $\E\big[e^{\kappa\|\y\|_{\H}^2}\big]$ nor do we estimate integrals of the form \eqref{mxdterp}.
		\item\textbf{Simpler and more direct proofs.} The absorption term allows direct control of $\int_{\H}\|\mathcal{A}\y\|_{\H}^2\d\eta$ and relative derivative bounds, which streamlines the argument that $\mathcal{N}_2$ is the closure of $\mathcal{N}_0$ and removes several lengthy interpolation/weighting estimates used in previous approaches. 
		\item\textbf{L\'evy noise-addtive case.} We treat Kolmogorov equation for the SCBF system under additive L\'evy noise. To the best of our knowledge, this is the first result of its kind, showing that the usual exponential-moment barrier can indeed be surmounted for jump noise models when the absorption hypothesis outlined above is employed.
		\item\textbf{Solving an infinite-dimensional HJB equation.} 
			We further apply our result to an infinite horizon stochastic optimal control problem, studying the solvability of the associated HJB equation. Due to the presence of  L\'evy noise, the HJB equation becomes an infinite-dimensional integro-differential equation. Using the  ``Carr\'e du Champ identity'' and the perturbation results, we prove the existence of solutions to this HJB equation by a fixed point argument. 
	\item\textbf{Broader applicability.} Because our assumptions on noise and moments are weaker, the methods developed here can be applied to a wider class of stochastic forcing terms and may be adaptable to other dissipative PDE with strong absorption-like nonlinearities.
	\end{enumerate}
	
	These features together constitute the main technical and conceptual novelty of the paper: \emph{The absorption term in \eqref{scbfjump} furnishes enough dissipation to replace exponential-moment techniques, allowing essential $m$-dissipativity of the Kolmogorov operator to be proved under significantly weaker and more natural assumptions, including the additive L\'evy noise setting.} Moreover, our findings provide a systematic method for analysing and solving an infinite-dimensional HJB equation associated with the SCBF system \eqref{scbfabs} and, more broadly, with L\'evy driven fluid models.
	
	\subsection{Organization of the article} The remainder of the paper is structured as follows. In the next section, we set up the functional framework used throughout the work. Section \ref{sec4} provides the necessary stochastic background and states the well-posedness result for the SCBF system \eqref{scbfabs} (see Theorem \ref{extunjump}). In Section \ref{secvarmea}, we prove the existence and uniqueness of the invariant measure (Theorem \ref{uniqinv}) for the transition semigroup $\{\P_t\}_{t\geq 0}$ associated with \eqref{scbfabs}, and we also obtain a bound for $\int_{\H}\|\mathcal{A}\x\|_{\H}^2\d\eta$ (Proposition \ref{lem5.2}). Section \ref{disscore} contains the proof of our main result (Theorem \ref{thm4.7}), which follows from an application of the Lumer–Phillips theorem. Finally, in Section \ref{sec5}, we illustrate an application of the main theorem to an infinite-horizon optimal control problem.

	\section{Mathematical framework}\label{sec2}\setcounter{equation}{0}
	This section introduces the function spaces essential for achieving the primary objectives of this work, following the framework given in \cite{gdp7} and \cite{JCR1, JCR}.

	\subsection{Sobolev space of periodic functions} 
	Let $\C_{\mathrm{p}}^{\infty}(\mathbb{T}^d;\R^d)$ denote the space of all infinitely differentiable $\R^d-$valued functions $\u$ defined on the $d-$dimensional torus $\mathbb{T}^d$, satisfying periodic boundary conditions of the form 
	$\u(x+e_{i},\cdot) = \u(x,\cdot)$, for all $x\in \R^d$ and $i=1,\ldots,d$, where $\{e_i\}_{i=1}^d$ is the canonical basis of $\R^d$. The periodic Sobolev space $\H_{\mathrm{p}}^s(\mathbb{T}^d):=\mathrm{H}_{\mathrm{p}}^s(\mathbb{T}^d;\mathbb{R}^d)$ is defined as the completion of $\C_{\mathrm{p}}^{\infty}(\mathbb{T}^d;\R^d)$  with respect to the Sobolev norm  $$\|\u\|_{{\H}^s_{\mathrm{p}}}:=\left(\sum_{0\leq|\boldsymbol\alpha|
		\leq s} \|\D^{\boldsymbol\alpha}\u\|_{\mathbb{L}^2(\mathbb{T}^d)}^2\right)^{1/2}.$$ 	
%	According to \cite[Proposition 5.39]{JCR1}, the space $\H_{\mathrm{p}}^s(\mathbb{T}^d)$, for $s\geq0$, can be equivalently characterized as
%	$$\left\{\u:\u=\sum_{k\in\mathbb{Z}^d}\y_{k}\mathrm{e}^{2\pi i k\cdot \xi /  \mathrm{L}},\ \overline{\u}_{k}=\u_{-k}, \  \|\u\|_{{\H}^s_f}:=\left(\sum_{k\in\mathbb{Z}^d}(1+|k|^{2s})|\u_{k}|^2\right)^{1/2}<\infty\right\}.$$ Furthermore, by \cite[Propositions 5.38]{JCR1} the norms $\|\cdot\|_{{\H}^s_{\mathrm{p}}}$ and $\|\cdot\|_{{\H}^s_f}$ are equivalent. 
	\begin{remark}
	Unlike in the NSE case, we do not assume that $\u$ has zero mean, because the absorption term $|\u|^{r-1}\u$ does not preserve this property (see \cite{KWH}). Therefore, the Poincar\'e inequality cannot be used, and we work with the full $\H^1$-norm.
	\end{remark}
	We consider the space 
	\begin{align*} 
	\mathcal{V}:=\{\u\in\C_{\mathrm{p}}^{\infty}(\mathbb{T}^d;\R^d):\nabla\cdot\u=0\},
	\end{align*}
	that is, the set of smooth, periodic, divergence free vector fields on the torus $\mathbb{T}^d$. 
	
	Let $\H$ and $\widetilde{\L}^{p}$ denote the closure of $\mathcal{V}$ in the spaces $\mathrm{L}^2(\mathbb{T}^d;\R^d)$ and $\mathrm{L}^p(\mathbb{T}^d;\R^d)$, respectively, for $p\in(2,\infty]$. Furthermore, We define the space $\V$ as the closure of the same set in the Sobolev space $\H^1_{\mathrm{p}}(\mathbb{T}^d)$. 
	
	The space $\H$ is equipped with the inner product $(\u,\v)=\int_{\mathbb{T}^d}\u(\xi)\cdot\v(\xi)\d \xi$ and the associated norm $\|\u\|_{\H}^2:=\int_{\mathcal{O}}|\u(x)|^2\d x$. On $\V$, we consider the inner product $(\u,\v)_{\V}:=(\u,\v)+(\nabla\u,\nabla\v)$ which induces the norm $	\|\u\|_{\V}^2:=\|\u\|_{\H}^2+\|\nabla\u\|_{\H}^2$. For $p\in(2,\infty)$, the space $\widetilde{\L}^{p}$ is endowed with the norm $\|\u\|_{\widetilde{\L}^p}^p:=\int_{\mathbb{T}^d}|\u(\xi)|^p\d \xi$, while for $p=\infty$, the norm on $\widetilde{\L}^{\infty}$ is given by
	$	\|\u\|_{\widetilde{\L}^{\infty}}:=\esssup_{\xi\in\mathbb{T}^d}|\u(\xi)|$.

	We denote $\langle \cdot,\cdot\rangle$ the duality pairing between $\V$  and its dual space $\V'$, as well as between $\widetilde{\L}^p$ and its dual $\widetilde{\L}^{p'}$, where the exponents $p$ and $p'$ satisfy the relation $\frac{1}{p}+\frac{1}{p'}=1$. Moreover, the Hilbert space $\H$ is naturally identified with its dual $\H'$. We equip the intersection space $\V\cap\widetilde{\L}^{p}$ with the norm $\|\u\|_{\V}+\|\u\|_{\widetilde{\L}^{p}},$ for $\u\in\V\cap\widetilde{\L}^p$. Its dual space, $\V'+\widetilde{\L}^{p'}$, is endowed with the norm $$\inf\left\{\max\left(\|\v_1\|_{\V'},\|\v_1\|_{\widetilde{\L}^{p'}}\right):\v=\v_1+\v_2, \ \v_1\in\V', \ \v_2\in\widetilde{\L}^{p'}\right\}.$$   
%	We first note that $\mathcal{V}\subset\V\cap\widetilde{\L}^{p}\subset\H$ and $\mathcal{V}$ is dense in $\H,\V$ and $\widetilde{\L}^{p},$ and hence $\V\cap\widetilde{\L}^{p}$ is dense in $\H$. 
%	We have the following continuous  embedding also:
%$$\V\cap\widetilde{\L}^{p}\hookrightarrow\H\equiv\H'\hookrightarrow\V'+\widetilde\L^{\frac{p}{p-1}}.$$ One can define equivalent norms on $\V\cap\widetilde\L^{p}$ and $\V'+\widetilde\L^{\frac{p}{p-1}}$ as  (see \cite{NAEG})
%\begin{align*}
%	\|\u\|_{\V\cap\widetilde\L^{p}}=\left(\|\u\|_{\V}^2+\|\u\|_{\widetilde\L^{p}}^2\right)^{1/2}\ \text{ and } \ \|\u\|_{\V'+\widetilde\L^{\frac{p}{p-1}}}=\inf_{\u=\v+\w}\left(\|\v\|_{\V'}^2+\|\w\|_{\widetilde\L^{\frac{p}{p-1}}}^2\right)^{1/2}.
%\end{align*}

 Apart from above spaces, the following functional spaces 
		(see \cite[Chapter 1]{gdp1}) are frequently used in the rest of the paper.
		\begin{enumerate}
				\item The set $\C_b(\H):=\C_b(\H;\R)$ denotes the Banach space of all bounded continuous functions  $\psi:\H\to\R$ endowed with the norm $\|\psi\|_{0}:=\sup\limits_{\y\in\H}|\psi(\y)|<+\infty.$ In a similar way, we can define the space $\C_b(\H;\H)$.
				
			 \item  
			%			For $\psi:\H\to\R$ and $\x,\h\in\H$, we set 
			%			$\big(\D_{\x}\psi(\x),\h\big)=\lim\limits_{\tau\to0}\frac{1}{\tau}
			%			[\psi(\x+\tau\h)-\psi(\x)]$, provided the limit exists.  
			We denote by $\C_b^1(\H)\subset\C_b(\H)$, the space of all bounded continuous functions $\psi:\H\to\R$ which are Fr\'echet differentiable on $\H$ with continuous and bounded derivative $\mathcal{D}_{\y}\psi$ such that
			$\|\psi\|_{1}:=\|\psi\|_{0}+\sup\limits_{\y\in\H}|\mathcal{D}_{\y}\psi(\y)|<+\infty.$
			
			\item Let $\mathscr{B}_b(\H)$ is the Banach space of all bounded and Borel measurable functions $\psi:\H\to\R$ endowed with the norm
			$\|\psi\|_{0}<+\infty$, for all $\psi\in\C_b(\H)$. 
%			Similarly, one can define the notion of $\mathscr{B}_b(\H;\H)$.
			
%			\item Let $\mathscr{B}_b(\H;\H)$ be the space of bounded Borel measurable functions $\mathbf{F}:\H\to\H$ such that 
%			$\|\mathbf{F}\|_{0}:=\sup_{\x\in\H}\|\mathbf{F}(\x)\|_{\H}<+\infty$, for $\mathbf{F}\in\C_b(\H;\H)$.
	\end{enumerate}

%	\subsubsection{$\mathrm{L}^2$-Hilbert spaces \cite{gdp7} }
	Let $\eta$ be an invariant measure.
	We denote by $\mathrm{L}^2(\H,\eta)$, the equivalence class of all Borel measurable functions $\varphi:\H\to\R$, which are square integrable and endowed with the inner product (\cite{gdp7})
	\begin{align*}
		(\varphi,\psi)_{\mathrm{L}^2(\H,\eta)}:=\int_{\H} \varphi(\x)\psi(\x)\eta(\d\x), \  \varphi,\psi\in\mathrm{L}^2(\H,\eta),
	\end{align*}
	and norm 
	\begin{align*}
		\|\varphi\|_{\mathrm{L}^2(\H,\eta)}:=\left(\int_{\H} |\varphi(\x)|^2\eta(\d\x)\right)^{\frac{1}{2}}, \ \ \varphi\in\mathrm{L}^2(\H,\eta).
	\end{align*}
	Let us now consider the space $\L^2(\H,\eta;\H)$, of all equivalence classes of Borel square integrable functions $\mathtt{F}:\H\to\H$ such that
	\begin{align*}
		\|\mathtt{F}\|_{\L^2(\H,\eta;\H)}:=\left(\int_{\H}\|\mathtt{F}(\x)\|_{\H}^2
		\eta(\d\x)\right)^{\frac{1}{2}} <\infty.
	\end{align*}
	Moreover, $\L^2(\H,\eta;\H)$ is a Hilbert space equipped with the following inner product 
	\begin{align*}
		(\mathtt{F},\mathtt{G})_{\L^2(\H,\eta;\H)}:=\int_{\H}
		 (\mathtt{F}(\x),\mathtt{G}(\x))\eta(\d\x), \ \ \ \mathtt{F},\mathtt{G}\in\L^2(\H,\eta;\H).
	\end{align*} 

%	The elements of $\L^2(\H,\eta;\H)$ are called as $\mathrm{L}^2$-\emph{vector fields} (cf. \cite{gdp7}). For any $F\in\L^2(\H,\eta;\H)$,  we can write $F(\x)=\sum\limits_{k=1}^{\infty}(F(\x),\boldsymbol{e}_k)\boldsymbol{e}_k$, $\eta$-a.e., where $\{\boldsymbol{e}_k\}_{k=1}^{\infty}$ is a complete orthonormal basis in $\H$.
	
	\subsection{Linear operator}
	Let $\mathscr{P} : \L^2(\mathbb{T}^d)\to\H$ denote the \emph{Helmholtz-Hodge orthogonal or Leray projection}. From \cite[Section 2.1]{JCR}, we infer that  it is bounded and self-adjoint operator. We define
	\begin{equation*}
		\left\{
		\begin{aligned}
			\mathcal{A}\u:&=-\mathscr{P}\Delta\u,\;\u\in\D(\mathcal{A}),\\ \D(\mathcal{A}):&=\V\cap\H^{2}_{\mathrm{p}}(\mathbb{T}^d).
		\end{aligned}
		\right.
	\end{equation*}
	Note that in view of Parseval's identity and the definition of 
$\|\cdot\|_{\H^2_\mathrm{p}}-$norm, one can show that 
$\H^2_\mathrm{p}(\mathbb{T}^d)=\D(\I+\A):=\V_2$.

	\subsection{Bilinear operator}
	Let $b(\cdot,\cdot,\cdot):\V\times\V\times\V\to\R$ be a continuous trilinear form defined by
\begin{align*}
	b(\u,\v,\w)=\int_{\mathbb{T}^d}(\u(\xi)\cdot\nabla)\v(\xi)\cdot\w(\xi)\d \xi.
\end{align*} 
By the Riesz representation theorem, we can define $\mathfrak{B}(\cdot,\cdot):\V\times\V\to\R$ a continuous bilinear operator such that $\langle\mathfrak{B}(\u,\v),\w\rangle=b(\u,\v,\w)$ for all $\u,\v,\w\in\V$, which also satisfies (see \cite {Te})
\begin{align}\label{syymB}
	\langle\mathfrak{B}(\u,\v),\w\rangle=-\langle\mathfrak{B}(\u,\w),\v\rangle \ \text{ and } \
	\langle\mathfrak{B}(\u,\v),\v\rangle=0,
\end{align}
for any $\u,\v,\w\in\V$. We also denote $\mathfrak{B}(\u)= \mathfrak{B}(\u, \u)$. 
%	\begin{lemma}\cite{MT2}
%		The trilinear map $b(\cdot,\cdot,\cdot):\V\times\V\times\V\to\R$ has a unique extension to a bounded trilinear map from $(\V\cap\wi\L^{r+1})\times(\V\cap\wi\L^{r+1})\times\V$ into $\R$. Moreover, for $r\geq1$, the bilinear operator
%		$\mathfrak{B}(\cdot)$ maps from $\V\cap\wi\L^{r+1}$ into $\V^{\prime}+\wi\L^{\frac{r+1}{r}}$ and satisfies 
%		\begin{align}\label{2.9a}
%		\|\mathfrak{B}(\u)\|_{\V'+\widetilde{\L}^{\frac{r+1}{r}}}\leq\|\u\|_{\widetilde{\L}^{r+1}}^{\frac{r+1}{r-1}}\|\u\|_{\H}^{\frac{r-3}{r-1}}\  \text{ and } \ 
%		\|\mathfrak{B}(\u)\|_{\V^{\prime}}\leq\|\u\|_{\widetilde{\L}^{4}}^{2},
%	\end{align}
%	for all $\u\in\V\cap\wi\L^{r+1}$.
%	\end{lemma}
%\begin{lemma}\cite{}
%	For every $r>3$ and for all $\u,\v\in\V\cap\wi\L^{r+1}$, the bilinear operator $\mathfrak{B}(\cdot):\V\cap\wi\L^{r+1}\to\V^{\prime}+\wi\L^{\frac{r+1}{r}}$ is locally Lipschitz and satisfies 
%	\begin{align*}
%		\|\mathfrak{B}(\u)-\mathfrak{B}(\v)\|_{\V'+\widetilde{\L}^{\frac{r+1}{r}}}\leq \left(\|\u\|_{\H}^{\frac{r-3}{r-1}}\|\u\|_{\widetilde{\L}^{r+1}}^{\frac{2}{r-1}}+\|\v\|_{\H}^{\frac{r-3}{r-1}}\|\v\|_{\widetilde{\L}^{r+1}}^{\frac{2}{r-1}}\right)\|\u-\v\|_{\widetilde{\L}^{r+1}},
%	\end{align*} 
%		Moreover, for $r=3$ and for all $\u,\v\in\wi\L^4$, we have
%			\begin{align*}
%		\|\mathfrak{B}(\u)-\mathfrak{B}(\v)\|_{\V'+\widetilde{\L}^{\frac{4}{3}}}\leq \left(\|\u\|_{\widetilde{\L}^4}+\|\v\|_{\widetilde{\L}^4}\right)\|\u-\v\|_{\widetilde{\L}^4}.
%			\end{align*}
%	\end{lemma}
	
  \begin{lemma}\cite{MT2}
  	For $r>3$, the bilinear operator $\mathfrak{B}(\cdot)$ satisfies the following monotonicity estimate
	 \begin{align}\label{3.4}
	 	|\langle\mathfrak{B}(\u)-\mathfrak{B}(\v),\u-\v\rangle|\leq
	 	\frac{\mu }{2}\|\nabla(\u-\v)\|_{\H}^2 +\frac{\beta}{4}\||\v|^{\frac{r-1}{2}}(\u-\v)\|_{\H}^2 +\varrho\|\u-\v\|_{\H}^2,
	 \end{align}
	 where 
	 \begin{align}\label{eqn-varrho}
	 \varrho:=\frac{r-3}{2\mu(r-1)}\left[\frac{4}{\beta\mu (r-1)}\right]^{\frac{2}{r-3}}.
	 \end{align}
	\end{lemma}

  \subsection{Nonlinear operator}
	Let $\mathfrak{C}(\y):=\mathscr{P}(|\y|^{r-1}\y)$ for $\y\in\V\cap\L^{r+1}$. In view of the \cite[Remark 1.6]{Te}, the projection operator $\mathscr{P}:\H^1\to\H^1$ is bounded. Consequently, the operator $\mathfrak{C}(\cdot):\V\cap\widetilde{\L}^{r+1}\to\V'+\widetilde{\L}^{\frac{r+1}{r}}$ is well-defined and we have $\langle\mathfrak{C}_1(\y),\y\rangle =\|\y\|_{\widetilde{\L}^{r+1}}^{r+1}.$ The map $\mathfrak{C}(\cdot):\widetilde{\L}^{r+1}\to\widetilde{\L}^{\frac{r+1}{r}}$ is Gateaux differentiable with Gateaux derivative given by (\cite{MT2})
	\begin{align}\label{2133}
		\mathfrak{C}'(\y)\z&=
		\left\{\begin{array}{cl}
			\mathscr{P}(\z),&\text{ for }r=1,\\ 
			\left\{
			\begin{array}{cc}\mathscr{P}(|\y|^{r-1}\z)+(r-1)
				\mathscr{P}\left(\frac{\y}{|\y|^{3-r}}(\y\cdot\z)\right),&\text{ if }\y\neq \mathbf{0},\\\mathbf{0},&\text{ if }\y=\mathbf{0},
			\end{array}
			\right.
			&\text{ for } 1<r<3,\\ \mathscr{P}(|\y|^{r-1}\z)+(r-1)\mathscr{P}(\y|\y|^{r-3}(\y\cdot\z)), &\text{ for }r\geq 3,
		\end{array}
		\right.
	\end{align}
	for all $\z\in\V\cap\widetilde{\L}^{r+1}$. 
%The following lemma pertains to some monotonicity estimates of the nonlinear operator $\mathfrak{C}(\cdot)$, which we used frequently throughout this work.
	\begin{lemma}[Monotonicity of $\mathfrak{C}(\cdot)$ (\cite{MT2})]
	For every $r\geq1$ and for all $\u,\v,\w\in\widetilde{\L}^{r+1}$, the nonlinear operator $\mathfrak{C}(\cdot)$ satisfies following estimates:
		\begin{align}\label{monoC1}
			\langle\mathfrak{C}(\u)-\mathfrak{C}(\v),\w\rangle\leq
			r\left(\|\u\|_{\widetilde{\L}^{r+1}}+\|\v\|_{\widetilde{\L}^{r+1}}\right)^{r-1}\|\u-\v\|_{\widetilde{\L}^{r+1}}\|\w\|_{\widetilde{\L}^{r+1}}.
		\end{align}
%		and 
%	   \begin{align}\label{monoC2}
%	   	\langle\mathfrak{C}(\u)-\mathfrak{C}(\v),\u-\v\rangle\geq \frac{1}{2}\||\u|^{\frac{r-1}{2}}(\u-\v)\|_{\H}^2+\frac{1}{2}\||\v|^{\frac{r-1}{2}}(\u-\v)\|_{\H}^2\geq\frac{1}{2^{r-1}}\|\u-\v\|_{\wi\L^{r+1}}^{r+1}.
%	   \end{align}
	\end{lemma}

	\begin{remark}[Essential estimates and calculations]
		As noted in the introduction, the main difficulty lies in estimating the convective term in appropriate norms. Fortunately, the rapidly growing nonlinear term, that is, the term $|\y|^{r-1}\y$ for $r \geq 3$, together with diffusion term $-\Delta\y$, dominate the convective term. Consequently, the later can be handled by means of H\"older’s and Young’s inequalities under suitable norms.
		We have the following equality on torus:
		\begin{align}\label{toruseq}
			(\mathfrak{C}(\y),\mathcal{A}\y)=\beta\||\y|^{\frac{r-1}{2}}\nabla\y\|_{\H}^{2} +4\beta\left[\frac{r-1}{(r+1)^2}\right]\|\nabla|\y|^{\frac{r+1}{2}}\|_{\H}^{2}.
		\end{align}
		Moreover, the following estimate is useful throughout the article :
		\begin{align}\label{syymB3}
			|(\mathfrak{B}(\u),\mathcal{A}\u)|\leq\frac{\mu}{2}\|\mathcal{A}\u\|_{\H}^2
			+\frac{\beta}{4}\||\u|^{\frac{r-1}{2}}\nabla\u\|_{\H}^2 +\varrho\|\nabla\u\|_{\H}^2,
		\end{align}
		where $\varrho$ is defined in \eqref{eqn-varrho}. 
	\end{remark}

	\section{The SCBF system perturbed by jump noise} \label{sec4}\setcounter{equation}{0}
	This section provides a brief overview of the stochastic framework used in our analysis and recall some known well-posedness result for the SCBF system. To this end, we first introduce the abstract functional framework. On projecting the first equation in \eqref{scbfjump}, we obtain
	\begin{equation}\label{scbfabs}
		\left\{
		\begin{aligned}
			\d\Y(t)&+[\mu \mathcal{A}\Y(t)+\mathfrak{B}(\Y(t))+\alpha\Y(t)+ \beta\mathfrak{C}(\Y(t))]\d t\\&=
			\sqrt{\Q}\d\W(t)+\int_{Z}\G(t,z)\wi{\pi}(\d t,\d z), \ t\in(0,T),\\
			\Y(0)&=\y\in\H,
		\end{aligned}
		\right.
	\end{equation}
	where $\Y(\cdot):=\mathscr{P}\u(\cdot)$, $\y:=\mathscr{P}\u_0$ and $\G(\cdot,\cdot)=\mathscr{P}\sigma(\cdot,\cdot)$. 
	\subsection{Stochastic setting}\label{Levy}
	We first review some stochastic preliminaries that will be used consistently throughout the article, drawing from \cite{SPJZ,PEP}.
	Let $\mathfrak{P}:=(\Omega,\mathscr{F},\mathbb{F},\mathbb{P})$, where $\mathbb{F}:=\{\mathscr{F}_t\}_{t\geq 0}$ is an increasing family of sub-sigma fields of $\mathscr{F}$ (containing all past events up to time $t$), be a stochastic basis. We assume that $\mathscr{F}$ equals to the $\sigma$-field generated by $\bigcup\limits_{t\geq0}\mathscr{F}_t$, denoted by $\mathscr{F}_\infty$. In addition, the stochastic basis $\mathfrak{P}$ satisfies the following usual hypothesis: 
	\begin{enumerate}
		\item [(i)] Every set $\F$ which belongs to the $\mathbb{P}$-completion of the $\sigma$-field $\mathscr{F}_\infty$ with $\mathbb{P}(\F)=0$ belongs to $\mathscr{F}_t$ for every $t\in[0,T]$.
		\item [(ii)] The family $\mathbb{F}$ is right continuous, that is,  $\mathscr{F}_t=\mathscr{F}_{t+}:=\bigcap\limits_{s>t}\mathscr{F}_s,$ for $0\leq t\leq T$.
	\end{enumerate} 
Let $\Q$ be a bounded, self-adjoint, and non-negative linear operator on $\H$ such that $\Tr(\Q)<\infty$. Let $\W$ denote an $\H$-valued $\Q$-Wiener process satisfying $\W(0)=0,$ $\mathbb{P}-$a.s.
We shall work under the following standing assumption throughout the article:
\begin{hypothesis}\label{trQ1}
	$\mathcal{A}^{\frac12}\Q^{\frac12}$ is a Hilbert-Schimdt operator.
\end{hypothesis}

Let us write $\Q_1:=\mathcal{A}^{\frac12}\Q\mathcal{A}^{\frac12}$. Since both $\mathcal{A}^{\frac12}$ and $\Q^{\frac12}$ are self-adjoint, Hypothesis \ref{trQ1} directly implies that 
\begin{align*}
	\Tr(\Q_1)=\Tr(\mathcal{A}^{\frac12}\Q\mathcal{A}^{\frac12})=
	\Tr(\mathcal{A}^{\frac12}\Q^{\frac12}\Q^{\frac12}\mathcal{A}^{\frac12})
	=\Tr\big(\mathcal{A}^{\frac12}\Q^{\frac12}\big(\mathcal{A}^{\frac12}\Q^{\frac12}\big)^{*}\big)
	=\|\mathcal{A}^{\frac12}\Q^{\frac12}\|_{\mathscr{L}_2(\H)}^2<\infty,
\end{align*}
where $\mathscr{L}_2(\H)$ is the Hilbert space of all Hilbert-Schmidt operators on $\H$.

	Let us suppose that $(Z,\mathscr{B}(Z))$ is a measurable space and let $\lambda$ be a positive $\sigma$-finite measure on it. Let $\pi:\Omega\times\mathscr{B}(\R^{+})\times\mathscr{B}(Z)\to\N\cup\{0\}$ be a time homogeneous Poisson random measure over $\mathfrak{P}$ with the intensity measure $\lambda$. The intensity measure $\lambda$ is given by  
\begin{align*}
	\E\left[\pi\left((0,1]\times\B\right)\right]=\lambda(\B), \text{ for all } \ \B\in\mathscr{B}(Z),
\end{align*}
and it satisfies the following conditions:
\begin{align*}
	\lambda(\{0\})=0 \ \text{ and } \ \ \int_{Z} (1\wedge|z|^2)\lambda(\d z)<+\infty.
\end{align*}
We denote by $\wi\pi:=\pi-\upgamma$, the compensated time homogeneous Poisson random measure  associated with $\lambda$, where the compensator measure $\upgamma$ is given by 
\begin{align*}
	\upgamma:\mathscr{B}(\R_+)\times\mathscr{B}(Z)\ni(I,\B)\mapsto\d(I)
	\lambda(\B)\in\R_+,
\end{align*} 
where $\d(\cdot)$ denotes the Lebesgue measure.

	Let $\D([0,T];\H)$ be the space of all $\H-$valued c\'adl\'ag functions defined $[0,T]$ equipped with the Skorohod $\J-$topology (see \cite{Ajakub,AJMM} for definition and its properties). Also let
	\begin{align*}
	\mathrm{L}^2(0,T;\mathrm{L}^2_{\lambda}(Z;\H))
		:=\mathrm{L}^2\big((0,T]\times Z;\mathscr{B}((0,T])\times\mathscr{B}(Z);\d t\otimes \lambda;\H\big),
	\end{align*}
	the space of all (equivalence class of) $\mathscr{B}((0,T])\times\mathscr{B}(Z)$-measurable functions $\G:[0,T]\times Z\to\H$ such that 
	\begin{align*}
		\int_0^T\int_{Z}\|\G(t,z)\|_{\H}^2\lambda(\d z)\d t
		<+\infty.
	\end{align*}
	Note that, if $\G\in\mathrm{L}^2(0,T;\mathrm{L}^2_{\lambda}(Z;\H))$, then the process 
	\begin{align*}
		\M(t):=\int_0^t\int_{Z}\G(s,z)\wi\pi(\d s,\d z),
	\end{align*} 
	is a c\'adl\'ag square integrable martingale (see \cite{SPJZ}) and the following It\^o isometry hold:
	\begin{align*}
		\E\bigg[\bigg\|\int_0^T\int_{Z}\G(s,z)\wi\pi(\d t,\d z) \bigg\|_{\H}^2\bigg]=
		\int_0^T\int_{Z}\|\G(t,z)\|_{\H}^2\lambda(\d z)\d t.
	\end{align*}
	Moreover, there exist increasing c\'adl\'ag processes, denoted by $[\M]_t$ and $\langle\M\rangle_t$, referred to as the \emph{quadratic variation process} and the \emph{Meyer process}, respectively, such that the difference $[\M]_t-\langle\M\rangle_t$ is a local martingale
	(see \cite[Section 1.6]{HKu} and \cite[Section 2, Chapter I]{Me}). Furthermore, for the process $\M(\cdot)$ defined above, one can verify the following (see \cite[Example 2.8]{UMS}):
	\begin{align*}
		[\M]_t= \int_0^t\int_{Z}\|\G(s,z)\|_{\H}^2\pi(\d s,\d z)\ \text{ and } \ \langle\M\rangle_t=\int_0^t\int_{Z}\|\G(s,z)\|_{\H}^2\lambda(\d z)
		\d s,
	\end{align*}
	for all $t\in[0,T]$. Indeed, we have $\E\big[\|\M(t)\|_{\H}^2\big]=\E\big[[\M]_t\big]=
		\E\big[\langle\M\rangle_t\big]$, so that 
	\begin{align*}
		\E\bigg[\int_0^t\int_{Z}\|\G(s,z)\|_{\H}^2\pi(\d s,\d z)\bigg]=
		\int_0^t\int_{Z}\|\G(s,z)\|_{\H}^2\lambda(\d z)\d s,
	\end{align*}
	for all $t\in[0,T]$. 
	
	%    Let $Y$ be a separable Banach space. Let  $\mathrm{F}_t:=\mathscr{F}_t\times\mathscr{B}(\R^+)\times\mathscr{B}(Z)$ be the product $\sigma$-field and $T>0$. We define the following class of measurable functions:
	%   \begin{align*}
		%  	\mathscr{H}(Z):=\bigg\{g:\R^+\times &Z\times\Omega\to Y\bigg| g \ \text{ is } \ \mathrm{F}_T/\mathscr{B}(Y)\text{-measurable and } \\ &g(t,z,\omega) \ \text{ is } \ \mathscr{F}_t-\text{adapted for every } \ t\in(0,T] \ \text{ and every} \ z\in Z \bigg\}.
		%   \end{align*}
	%  Let us define the following equivalence class of measurable functions (see \cite{}):
	%  \begin{align*}
		%  	\mathrm{L}^q_{\lambda,T}(Z;Y)&:=\left\{\xi:Z\to Y\bigg| 	|\xi|_{\mathrm{L}^q_{\lambda,T}(Z;Y)}:=\left(\int_{Z}\|\xi(z)\|_{Y}^q \lambda(\d z)\right)^{\frac{1}{q}}<\infty\right\},\\
		%  	\mathrm{L}^q(\Omega;Y)&:=\left\{\xi:\Omega\to Y\bigg| 	|\xi|_{\mathrm{L}^q(\Omega;Y)}:=\left(\E(\|\xi\|_{Y}^q)\right)^{\frac{1}{q}}<\infty\right\},\\
		%  \mathrm{L}^q(\Omega\times[0,T];Y)&:=\bigg\{\xi:\Omega\times[0,T]\to Y\bigg| 	|\xi|_{\mathrm{L}^q(\Omega\times[0,T];Y)}:=\left(\E\int_0^T\|\xi(t)\|_{Y}^q \d t\right)^{\frac{1}{q}}<\infty\bigg\},
		%  \end{align*}
	
	\subsection{Well-posedness result} 
	We now recall some fundamental existence and uniqueness results for SCBF system perturbed by L\'evy noise. Let us first provide the definition of a unique global strong solution to the system \eqref{scbfabs}.
	\begin{definition}\label{defjumpe}
	 Let $\x\in\H$ be given. An $\H-$valued $\mathscr{F}_t-$adapted stochastic process $\X(\cdot)$ is called a \emph{strong solution} (or \emph{probabilistically strong solution}) to the system \eqref{scbfabs} if the following conditions are satisfied:
	 \begin{itemize}
	 	\item The process 
	 	\begin{align*}
	 		\Y(\cdot)\in\mathrm{L}^2\left(\Omega;\mathrm{L}^{\infty}(0,T;\H)\cap
	 		\mathrm{L}^2(0,T;\V)\right)\cap\mathrm{L}^{r+1}
	 		\big(\Omega;\mathrm{L}^{r+1}(0,T;\wi\L^{r+1})\big),
	 	\end{align*}
	 	 and $\Y(\cdot)$ has a $\V\cap\wi\L^{r+1}-$valued modification, which is progressively measurable with c\'adl\'ag paths in $\H$ and 
	 	 \begin{align*}
	 	 	\Y(\cdot)\in\D([0,T];\H)\cap\mathrm{L}^2(0,T;\V)\cap\mathrm{L}^{r+1}
	 	 	(0,T;\wi\L^{r+1}), \ \mathbb{P}-\text{a.s.}
	 	 \end{align*}
	 	 \item The following equality holds for every $t\in[0,T]$, as an element of $\V^{\prime}+\wi\L^{\frac{r+1}{r}}$, $\mathbb{P}-$a.s.
	 	 \begin{align*}
	 	 	\Y(t)=\y&-\int_0^t[\mu\mathcal{A}\Y(s)+\mathfrak{B}(\Y(s))+\alpha\Y(s)+
	 	 	\beta\mathfrak{C}(\Y(s))]\d s+
	 	 	\int_0^t\sqrt{\Q}\d\W(s)
	 	 	\nonumber\\&\quad+\int_0^t\int_{Z}\G(s,z)\wi\pi(\d s,\d z).
	 	 \end{align*}
	 	 Alternatively, for any $\v\in\V\cap\wi\L^{r+1}$, $\mathbb{P}-$a.s., we have following:
	 	 \begin{align*}
	 	 (\Y(t),\v)&=(\y,\v)-\int_0^t\langle\mu\mathcal{A}\Y(s)+ \alpha\Y(s)+\mathfrak{B}(\Y(s))+\beta\mathfrak{C}(\Y(s)),\v\rangle\d s \nonumber\\&\quad+
	 	 \int_0^t\left(\sqrt{\Q}\d \W(s),\v\right)+\int_0^t\int_{Z}(\G(s,z),\v)
	 	 \wi\pi(\d s,\d z).
	 	 \end{align*}
	 	 \item The following It\^o formula holds, for all $t\in[0,T]$, $\mathbb{P}-$a.s.:
	 	 \begin{align}\label{engeqjump}
	 	 	&\|\Y(t)\|_{\H}^2+2\mu\int_0^t \|\nabla\Y(s)\|_{\H}^2\d s+
	 	 	2\alpha\int_0^t\|\Y(s)\|_{\H}^2\d s+2\beta\int_0^t \|\Y(s)\|_{\wi\L^{r+1}}\d s
	 	 	\nonumber\\&
	 	 	=\|\y\|_{\H}^2+t\Tr(\Q)+\int_0^t\|\G(s,z)\|_{\H}^2\pi(\d s,\d z)
	 	 	\nonumber\\&\quad+
	 	 	2\int_0^t(\sqrt{\Q}\d\W(s),\Y(s))+
	 	 	  2\int_0^t\int_{Z}\G(s,z)\wi\pi(\d s,\d z).
	 	 \end{align}
	\end{itemize}
	\end{definition}
	
	\begin{definition}\label{defjumpu}
	A strong solution $\Y(\cdot)$ to the system $\eqref{scbfabs}$ is called a \emph{pathwise unique strong solution}  if
	$\widetilde{\Y}(\cdot)$ is an another strong
	solution, then $$\mathbb{P}\big\{\omega\in\Omega:\Y(t)=\widetilde{\Y}(t),\ \text{ for all }\ t\in[0,T]\big\}=1.$$ 
	\end{definition}
	
	\begin{theorem}\label{extunjump}\cite[Theorem 3.6]{MT2}
		Let  $\x\in\H$ be given. Assume that 
		\begin{align*}
		\mathpzc{K}:=\int_0^T\int_{Z}\|\G(t,z)\|_{\H}^2\lambda(\d z)\d t<+\infty.
		\end{align*}
		For $r>3$, there exists a \emph{pathwise unique strong solution} to the system \eqref{scbfabs} in the sense of Definition \ref{defjumpe}-\ref{defjumpu}. Moreover, the following energy estimate holds:
		\begin{align*}
		&\E\left[\sup\limits_{t\in[0,T]}\|\Y(t)\|_{\H}^2+\int_0^T
		\|\nabla\Y(t)\|_{\H}^2\d t+\int_0^T\|\Y(t)\|_{\H}^{2}\d t+
		\int_0^T\|\Y(t)\|_{\wi\L^{r+1}}^{r+1}\d t \right]
		\nonumber\\&\leq
		 C(\|\y\|_{\H},\Tr(\Q),\mu,\alpha,\beta,T,\mathpzc{K}).
		\end{align*}
	\end{theorem}
		 
		 \subsection{Estimates for bounds of derivatives}
		We now estimate the bounds of derivatives for the solution $\Y(\cdot)$ of the SCBF system \eqref{scbfabs}, which is useful while proving the essential $m-$dissipativity of the Kolmogorv operator. Particularly, we find an estimate for $\E\left[\|\boldsymbol{\xi}^{\boldsymbol{h}}(t,\y)\|_{\H}^2\right]$, where $\boldsymbol{\xi}^{\boldsymbol{h}}(t,\y)=\mathcal{D}_{\y}\Y(t,\y)\boldsymbol{h},$  for all $ \y, \boldsymbol{h}\in\H$.
		 \begin{lemma}\label{lem4.4}
		 	Let us assume that 
		 	\begin{align}\label{419}
		 		\alpha>\varrho_1:=\frac{r-3}{2\mu(r-1)}\left[\frac{2}{\beta\mu (r-1)}\right]^{\frac{2}{r-3}}.
		 	\end{align}
		 	 Then, we have 
		 	\begin{align}\label{4p20}
		 	\|\boldsymbol{\xi}^{\boldsymbol{h}}(t,\y)\|_{\H}\leq \|\boldsymbol{h}\|_{\H}e^{-\mathpzc{k}_1t}, \ \text{ for all } \ t\in[0,T],
		 	\end{align}
		 where 
		 	\begin{equation}\label{kappalpha1}
		 	\mathpzc{k}_1=
		 	\left\{
		 	\begin{aligned}
		 		\alpha-\varrho_1, \ &\text{ when } r>3,\\
		 		\alpha, \ &\text{ when } r=3 \ \text{ with } 2\beta\mu\geq1.
		 	\end{aligned}
		 	\right.
		 \end{equation}
		 \end{lemma}
		 \begin{proof}
		 	We write $\boldsymbol{\xi}^{\boldsymbol{h}}(t)$ instead of $\boldsymbol{\xi}^{\boldsymbol{h}}(t,\y),$ for simplicity. Then $\boldsymbol{\xi}^{\boldsymbol{h}}(\cdot)$ satisfies 	for a.e. $t\in[0,T]$,  $\mathbb{P}$-a.s.: 
		 	\begin{equation}\label{4.21}
		 		\left\{
		 		\begin{aligned}
		 			\frac{\d\boldsymbol{\xi}^{\boldsymbol{h}}(t)}{\d t}+\mu\mathcal{A}\boldsymbol{\xi}^{\boldsymbol{h}}(t)+\alpha\boldsymbol{\xi}^{\boldsymbol{h}}(t)+\mathfrak{B}'(\X(t))\boldsymbol{\xi}^{\boldsymbol{h}}(t)+\beta\mathfrak{C}'(\X(t))\boldsymbol{\xi}^{\boldsymbol{h}}(t)&=\boldsymbol{0},\\
		 			\boldsymbol{\xi}^{\boldsymbol{h}}(0)&=\h, 
		 		\end{aligned}\right. 
		 	\end{equation} 
		 	where $\mathfrak{B}'(\Y)\boldsymbol{\xi}^{\boldsymbol{h}}=\mathfrak{B}(\Y,\boldsymbol{\xi}^{\boldsymbol{h}})+\mathfrak{B}(\boldsymbol{\xi}^{\boldsymbol{h}},\Y)$ and $\mathfrak{C}'(\Y)$ is defined in \eqref{2133}. Taking the inner product with $\boldsymbol{\xi}^{\boldsymbol{h}}(\cdot)$ to the first equation in \eqref{4.21}, we find 
		 	\begin{align}\label{4p22}
		 		&	\frac{1}{2}\frac{\d}{\d t} \|\boldsymbol{\xi}^{\boldsymbol{h}}(t)\|_{\H}^2 +\mu\|\nabla\boldsymbol{\xi}^{\boldsymbol{h}}(t)\|_{\H}^2
		 		+\alpha\|\boldsymbol{\xi}^{\boldsymbol{h}}(t)\|_{\H}^2+
		 		\beta\||\Y(t)|^{\frac{r-1}{2}}\boldsymbol{\xi}^{\boldsymbol{h}}(t)\|_{\H}^2
		 		\nonumber\\&\quad+\beta(r-1) \||\Y(t)|^{\frac{r-3}{2}}(\Y(t)\cdot\boldsymbol{\xi}^{\boldsymbol{h}}(t))\|_{\H}^2\nonumber\\&=-\big(\mathfrak{B}(\boldsymbol{\xi}^{\boldsymbol{h}}(t),\Y(t)),\boldsymbol{\xi}^{\boldsymbol{h}}(t)\big).
		 	\end{align}
		     By using the properties \eqref{syymB} of bilinear operator and applying the Cauchy Schwarz inequality, we find
		     \begin{align}\label{Best1}
		     \big|\big(\mathfrak{B}(\boldsymbol{\xi}^{\boldsymbol{h}}(t),\Y(t)),\boldsymbol{\xi}^{\boldsymbol{h}}(t)\big)\big|= \big|\big(\mathfrak{B}(\boldsymbol{\xi}^{\boldsymbol{h}}(t),
		     \boldsymbol{\xi}^{\boldsymbol{h}}(t)),\Y(t)\big)\big|
		     \leq
		     \frac{\mu}{2}\|\nabla\boldsymbol{\xi}^{\boldsymbol{h}}(t)\|_{\H}^2+
		     \frac{1}{2\mu}\|\Y(t)\boldsymbol{\xi}^{\boldsymbol{h}}(t)\|_{\H}^2.
		     \end{align}
		     For $r>3$, by the application of H\"older's and Young's inequalities, we estimate the last term in \eqref{Best1} as follows
		     \begin{align}\label{Best2}
		     	\|\Y(t)\boldsymbol{\xi}^{\boldsymbol{h}}(t)\|_{\H}^2&=
		     	\int_{\mathbb{T}^d}|\Y(t,x)|^2|\boldsymbol{\xi}^{\boldsymbol{h}}(t,x)|^2
		     	\d x
		     	\nonumber\\&=
		     	\int_{\mathbb{T}^d} |\Y(t,x)|^2
		     	|\boldsymbol{\xi}^{\boldsymbol{h}}(t,x)|^{\frac{4}{r-1}}
		     	|\boldsymbol{\xi}^{\boldsymbol{h}}(t,x)|^{\frac{2(r-3)}{r-1}}\d x
		     	\nonumber\\&\leq
		     	\left(\int_{\mathbb{T}^d} |\Y(t,x)|^{r-1}
		     	|\boldsymbol{\xi}^{\boldsymbol{h}}(t,x)|^{2}\d x\right)^{\frac{2}{r-1}}
		     	\left(\int_{\mathbb{T}^d}
		     	|\boldsymbol{\xi}^{\boldsymbol{h}}(t,x)|^{2}\d x\right)^{\frac{r-3}{r-1}}
		     	\nonumber\\&\leq
		     	\beta\mu\||\Y(t)|^{\frac{r-1}{2}}\boldsymbol{\xi}^{\boldsymbol{h}}(t)\|_{\H}^2
		     	+\varrho_1\|\boldsymbol{\xi}^{\boldsymbol{h}}(t)\|_{\H}^2.
		     \end{align}
		     On substituting \eqref{Best1}-\eqref{Best2} into \eqref{4p22} and simplifying, we obtain
		     \begin{align}\label{4p22bd}
		     	\frac{1}{2}\frac{\d}{\d t}\|\boldsymbol{\xi}^{\boldsymbol{h}}(t)\|_{\H}^2 
		     	+\alpha\|\boldsymbol{\xi}^{\boldsymbol{h}}(t)\|_{\H}^2
		     	\leq
		     	\varrho_1\|\boldsymbol{\xi}^{\boldsymbol{h}}(t)\|_{\H}^2.
		     \end{align}
		    An application of Gr\"onwall's inequality in \eqref{4p22bd} gives
		 	\begin{align}\label{424}
		 		\|\boldsymbol{\xi}^{\boldsymbol{h}}(t)\|_{\H}^2\leq \|\boldsymbol{h}\|_{\H}^2e^{-2\left(\alpha-\varrho_1\right)t},
		 	\end{align}
		 	for all $t\in[0,T]$. 
		 	
		 	For $r=3$, from \eqref{4p22} and \eqref{Best1}, we immediately have
		 	\begin{align*}
		 		&\frac{1}{2}\frac{\d}{\d t} \|\boldsymbol{\xi}^{\boldsymbol{h}}(t)\|_{\H}^2 +\frac{\mu}{2}\|\nabla\boldsymbol{\xi}^{\boldsymbol{h}}(t)\|_{\H}^2
		 	+\alpha\|\boldsymbol{\xi}^{\boldsymbol{h}}(t)\|_{\H}^2+
		 	\left(\beta-\frac{1}{2\mu}\right)\||\Y(t)|^{\frac{r-1}{2}}\boldsymbol{\xi}^{\boldsymbol{h}}(t)\|_{\H}^2
		 	\nonumber\\&\quad+\beta(r-1) \||\Y(t)|^{\frac{r-3}{2}}(\Y(t)\cdot\boldsymbol{\xi}^{\boldsymbol{h}}(t))\|_{\H}^2
		 	\leq0.
		 	\end{align*}
		 	Therefore, for $2\beta\mu\geq1$, we conclude
		 	\begin{align*}
		 	\|\boldsymbol{\xi}^{\boldsymbol{h}}(t,\y)\|_{\H}\leq \|\boldsymbol{h}\|_{\H}e^{-\alpha t},
		 	\end{align*}
		 	which completes the proof. 
		 \end{proof}
		 
		 \section{Existence and uniqueness of invariant measure}\label{secvarmea}\setcounter{equation}{0}
		In this section, we recall some known results concerning the existence and uniqueness of invariant measure for the transition semigroup associated with the SCBF system described in \eqref{scbfabs} (see \cite[Section 5]{MT2}). We begin by defining the corresponding transition semigroup. 
		
		 For any $\uppsi\in\mathscr{B}_b(\H)$ and $t\geq0$, we introduce the family of operators $\{\mathtt{P}_t\}_{t\geq0}$ by
		 \begin{align}\label{tight1}
		 	(\mathtt{P}_t\uppsi)(\y)=\E\left[\uppsi(\Y(t,\y))\right], \ \y\in\H,
		 \end{align} 
		 where $\Y=\Y(t,\y)$ denotes the unique strong solution to the SCBF system \eqref{scbfabs} with initial condition $\y$.
%		 {\color{Maroon} Give the proper reasoning that how $\mathtt{P}_t$ is well-defined and why it is a Markov process !!}
		 
		 Moreover, for every $\uppsi\in\C_b(\H)$, the mapping $\H\ni\y\mapsto\E\left[\uppsi(\Y(t,\y))\right]\in\R$ is continuous $\H$. Consequently, the family $\{\mathtt{P}_t\}_{t\geq0}$ defines a Feller semigroup. However, in general, this semigroup is not a strongly continuous on $\C_b(\H)$. Under suitable assumptions, one can show that $\{\mathtt{P}_t\}_{t\geq0}$ is instead \emph{weakly continuous} on $\C_b(\H)$ (see \cite[Appendix B, pp. 305]{SC1} for the definition and properties).
		 Following \cite{SC1}, the infinitesimal generator of this weakly continuous semigroup is defined as the unique closed linear operator $\mathcal{N}:\D(\mathcal{N})\subset\C_b(\H)\to\C_b(\H)$ characterized by the resolvent relation 
		 \begin{align*}
		 	\mathcal{R}(\lambda,\mathcal{N})\uppsi(\y)=\int_0^{\infty} e^{-\lambda t} \mathtt{P}_t \uppsi(\y)\d t, 
		 \end{align*}
	 for any $\lambda>0,$  $\uppsi\in\C_b(\H).$ and $\y\in\H$. 
		 % We will study the Kolmogorov equations in the space $\mathrm{L}^2(\H;\eta)$, where $\eta$ is an invariant measure for $\{\mathtt{P}_t\}_{t\geq0}$, that is $\eta$ is a Borel probability measure in $\H$ such that 
		 % \begin{align*}
		 	% 	\int_{\H}\mathtt{P}_t\uppsi \d\eta=\int_{\H}\uppsi\d\eta, \ \text{ for all }\ \uppsi\in\C_b(\H). 
		 	% \end{align*}
		 
		 \subsection{Existence of an invariant measure}\label{sec4.2}
		 It is well known that the system \eqref{scbfabs} has a unique strong solution $\Y(\cdot)$ with paths in $\D([0,T];\H)\cap\mathrm{L}^2(0,T;\V)\cap\mathrm{L}^{r+1}(0,T;\wi\L^{r+1})$, $\mathbb{P}$-a.s. Taking the expectation in the energy equality (or It\^o formula) \eqref{engeqjump}, and noting that the last two terms in \eqref{engeqjump} are martingales with zero expectation, we obtain 
		 \begin{align}\label{4.26}
		 	&	\E\left[\|\Y(t)\|_{\H}^2+2\mu\int_0^t\|\nabla\Y(s)\|_{\H}^2\d s+ 2\alpha\int_0^t\|\Y(s)\|_{\H}^2\d s+ 2\beta\int_0^t\|\Y(s)\|_{\wi\L^{r+1}}^{r+1}\d s\right]\nonumber\\& =\|\x\|_{\H}^2+\Tr(\Q)t+\int_0^t\int_{Z}\|\G(s,z)\|_{\H}^2\lambda(\d z)\d s,
		 \end{align}
		 for all $t\in[0,T]$. Consequently, we have
		 \begin{align}\label{varmeas}
		 \frac{2\min\{\mu,\alpha\}}{t}\E\left[\int_0^t\|\Y(s)\|_{\V}^2\d s\right]  \leq\frac{1}{t_0}\|\x\|_{\H}^2+\Tr(\Q)+\frac{1}{t_0}\mathpzc{K}, \ \text{ for all } \ t>t_0,
		 \end{align}
		 where $\mathpzc{K}=\int_0^T\int_{Z}\|\G(t,z)\|_{\H}^2\lambda(\d z)\d t< +\infty$.  Since $\Tr(\Q)<+\infty$, by applying Markov's inequality and using \eqref{varmeas}, we obtain the following limit:
		 \begin{align}\label{tight}
		 	\lim\limits_{r\to\infty}\sup\limits_{t>t_0}\left[\frac{1}{t}\int_0^t \mathbb{P}\{\|\Y(s)\|_{\V}>R\}\d s\right]
		 	\leq\lim\limits_{R\to\infty}\sup\limits_{t>t_0}\frac{1}{R^2}
		 	\E\left[\frac{1}{t}\int_0^t \|\Y(s)\|_{\V}^2\d s\right]=0.
		 \end{align}
		 Let us set
		 \begin{align*}
		 	\zeta_{t,\y}(\cdot)=\frac{1}{t}\int_0^t \Pi_{s,\y}(\cdot)\d s,
		 \end{align*}
		 where, $\Pi_{t,\y}(\Lambda)=\mathbb{P}\{\Y(t,\y)\in\Lambda\}, \ \Lambda\in\mathscr{B}(\H)$, is the law of $\Y(t,\y)$ for each $\y\in\H$. Hence, along with the estimate in \eqref{tight} and compact embedding $\V\hookrightarrow\H$, it is clear that the sequence of probability measures $\{\zeta_{t,\y}\}_{t>0}$ is tight.
		Therefore, from Krylov-Bogoliubov theorem (cf. \cite[Theorem 3.1.1, Chapter 3]{gdp2}), there exists an invariant measure $\eta$ for the transition semigroup $\{\mathtt{P}_t\}_{t\geq0}$.
%		 \begin{align*}
%		 	\int_{\H}(\mathtt{P}_t\upvarphi)(\x)\eta(\d \x)=\int_{\H}\upvarphi(\x)\eta(\d \x), \ \text{ for all }\ \upvarphi\in\C_b(\H). 
%		 \end{align*}
		 
		 \subsection{Uniqueness of invariant measures} 
		The uniqueness of the invariant measure is a direct consequence of the exponential stability of the system \eqref{scbfabs}. Although \cite{MT2} already established this result, we present a complete proof here for completeness. Note that SCBF system \eqref{scbfabs} is perturbed by both additive  L\'evy noise. Consequently, the exponential stability of the system \eqref{scbfabs} follows directly from the monotonicity estimates \eqref{3.4} and \eqref{monoC1}, together with Gr\"onwall's inequality. Therefore, we omit the proof of the following proposition, which establishes the exponential stability of the system \eqref{scbfabs}.
		 \begin{proposition}\label{expstbvar}
		 	Let $\Y_1(\cdot)$ and $\Y_2(\cdot)$ be two solutions of the system \eqref{scbfabs} with initial conditions $\Y_1(0)=\y_1$ and $\Y_2(0)=\y_2$, where $\y_1,\y_2\in\H$, respectively. Then, for $r>3$ with $\alpha>\varrho$ and $r=3$ with $2\beta\mu\geq1$, we have 
		 	\begin{align*}
		 		\|\Y_1(t)-\Y_2(t)\|_{\H}\leq e^{-\mathpzc{k}_2t}\|\y_1-\y_2\|_{\H}, \ \mathbb{P}-
		 		\text{ a.s. },
		 	\end{align*}
		 	where
		 	\begin{equation}\label{kappalpha}
		 		\mathpzc{k}_2=
		 		\left\{
		 		\begin{aligned}
		 			\alpha-\varrho, \ &\text{ when } r>3,\\
		 			\alpha, \ &\text{ when } r=3 \ \text{ with } 2\beta\mu\geq1,
		 		\end{aligned}
		 		\right.
		 	\end{equation}
		 	where $\varrho$ is defined in \eqref{eqn-varrho}. 
		 \end{proposition}
		 		
\begin{theorem}\label{uniqinv}
Under the assumptions of Proposition \ref{expstbvar}, the invariant measure $\eta$ for $\{\mathtt{P}_t\}_{t\geq 0}$ is unique. 
\end{theorem}

\begin{proof}
Let $\upvarphi\in\mathrm{Lip}_b(\H)$. By using the fact that $\eta$ is a invariant measure along with Proposition \ref{expstbvar}, we estimate
\begin{align}\label{expfast}
\left|(\mathtt{P}_t\upvarphi)(\x_1)-\int_{\H}\upvarphi(\x_2)\eta(\d\x_2)\right|
&=
\left|\E\big[\upvarphi(\Y(t,\x_1))\big]-\int_{\H}\mathtt{P}_t\upvarphi(\x_2)\eta(\d\x_2)
\right| 
\nonumber\\&\leq
\left|\int_{\H}\E\big[\upvarphi(\Y(t,\x_1))-\upvarphi(\Y(t,\x_2))\big]\eta(\d\x_2)
\right| 
\nonumber\\&\leq
\mathrm{Lip}(\upvarphi)\int_{\H}\E\big[\|\Y(t,\x_1)-\Y(t,\x_2)\|_{\H}\big]
\eta(\d\x_2)
\nonumber\\&\leq
\mathrm{Lip}(\upvarphi)e^{-\mathpzc{k}_2 t}\int_{\H}\|\x_1-\x_2\|_{\H}
\eta(\d\x_2)
\nonumber\\&\leq
\mathrm{Lip}(\upvarphi)e^{-\mathpzc{k}_2 t}\bigg[\|\x_1\|_{\H}+
\underbrace{\int_{\H}\|\x_2\|_{\H}\eta(\d\x_2)}_{<+\infty}\bigg]
\nonumber\\&\to0 \ \text{ as } t\to\infty.
\end{align}
Since, $\mathrm{Lip}_b(\H)$ is dense in $\C_b(\H)$, thus we have shown that 
\begin{align*}
\lim\limits_{t\to\infty}(\mathtt{P}_t\upvarphi)(\x_1)=
\int_{\H} \upvarphi(\x_2)\eta(\d\x_2), \ \eta-\text{ a.s., } 
\end{align*}
for all $\x_1\in\H$ and $\upvarphi\in\C_b(\H)$. Limit \eqref{expfast} shows that $\mathtt{P}_t\upvarphi(\x_1)$ converges exponentially fast to equilibrium, which is the exponential mixing property. Let us now establish the uniqueness of the invariant measure. Let $\eta$ and $\tilde\eta$ be two invariant measures for the transition semigroup $\{\mathtt{P}_t\}_{t\geq 0}$. Then, proceeding in a similar manner as we did in \eqref{expfast}, we find for all $\upvarphi\in\C_b(\H)$
\begin{align*}
&\left|\int_{\H}\uppsi(\x_1)\eta(\d\x_1)-\int_{\H}\uppsi(\x_2)
\tilde\eta(\d\x_2)\right|
\nonumber\\&=
\left|\int_{\H}\mathtt{P}_t\uppsi(\x_1)\eta(\d\x_1)-\int_{\H}\mathtt{P}_t\uppsi(\x_2)
\tilde\eta(\d\x_2)\right|
\nonumber\\&=
\left|\int_{\H}\int_{\H}\big[\mathtt{P}_t\uppsi(\x_1)-\mathtt{P}_t\uppsi(\x_2)\big]
\eta(\d\x_1)\tilde\eta(\d\x_2)\right|
\nonumber\\&\leq
\mathrm{Lip}(\upvarphi)e^{-\mathpzc{k}_2 t}\int_{\H}\int_{\H}
\|\x_1-\x_2\|_{\H}\eta(\d\x_1)\tilde\eta(\d\x_2)
\nonumber\\&
\to0,  \ \text{ as } t\to\infty,
\end{align*}
Hence, we obtain $\eta=\tilde\eta$.
\end{proof}

\subsection{Transition semigroup in $\mathrm{L}^2(\H,\eta)$} In what follows, we assume that $\mathrm{G}$ does not depend on $t$. Let $\uppsi\in\C_b(\H)$. By using Jensen's inequality along with the invariance of $\eta$, we deduce the following:
\begin{align}\label{extd}
\int_{\H} (\mathtt{P}_t\uppsi(\x))^2\eta(\d\x)\leq\int_{\H} \uppsi^2(\x)\eta(\d\x).
\end{align}
Since $\C_b(\H)$ is dense in $\mathrm{L}^2(\H,\eta)$ (\cite{gdp1}), an application of the Lebesgue dominated convergence theorem ensures that the family $\{\mathtt{P}_t\}_{t\geq0}$ admits a unique extension to a strongly continuous contraction semigroup on $\mathrm{L}^2(\H,\eta)$, which we continue to denote by $\{\mathtt{P}_t\}_{t\geq0}$ (cf. \cite{gdp1,gdp7}). We denote by $\mathcal{N}_2: \D(\mathcal{N}_2)\subset\mathrm{L}^2(\H;\eta)\to\mathrm{L}^2(\H;\eta)$ the infinitesimal generator associated with this semigroup.

We introduce algebra of \emph{exponential functions}: 
\begin{align*}
\mathscr{E}_{\A}(\H):=\mathrm{linspan}\{\uppsi_{\h}(\cdot)=e^{i(\h,\cdot)}: \h\in\D(\mathcal{A})\}, \ i=\sqrt{-1}.
\end{align*}
 From \cite[Proposition 1.2, Chapter 1, pp. 7]{gdp1}, the space $\mathscr{E}_{\mathcal{A}}(\H)$ is dense in $\mathrm{L}^2(\H;\eta)$.  Let us define on $\mathscr{E}_{\mathcal{A}}(\H)$, the following Kolmogorov integro-differential operator:
\begin{align}\label{4p5}
(\mathcal{N}_0\uppsi)(\x)=&\frac{1}{2}\Tr\left[\Q\D_{\x}^2\uppsi(\x)\right]-(\mu\mathcal{A}\x+\alpha\x+\mathfrak{B}(\x)+\beta\mathfrak{C}(\x),\D_{\x}\uppsi(\x))
\nonumber\\&\quad+
\int_{Z}\left[\uppsi(\x+\G(z))-\uppsi(\x)-(\G(z),\D_{\x}\uppsi(\x))\right]\lambda(\d z) , \ \text{ for all }\ \x\in\H.
\end{align}
 \emph{One of the central challenges in this work is to elucidate the relationship between the infinitesimal generator $\mathcal{N}_2$ and the Kolmogorov operator $\mathcal{N}_0$ defined in \eqref{4p5}.}
By an  application of It\^o's formula, it follows that $$\mathcal{N}_2\uppsi=\mathcal{N}_0\uppsi,$$ for all $\uppsi\in\mathscr{E}_{\A}(\H)$ (see Step-2 of Theorem \ref{thm4.7}). The primary objective is to show that $\mathcal{N}_2$ coincides with the closure of $\mathcal{N}_0$, equivalently, that $\mathscr{E}_{\A}(\H)$ forms a core for $\mathcal{N}_2$. In this case, $\mathcal{N}_0$ is said to be \emph{essentially $m$-dissipative} (see \cite{VBGA}). 

A key implication of this property is the well-posedness of the elliptic problem
\begin{align}\label{frd1}
	\kappa\uppsi-\mathcal{N}_0\uppsi=g, \ \kappa>0,
\end{align}
in the sense of Friedrichs. More precisely, for every $\kappa>0$ and each $g\in\mathrm{L}^2(\H;\eta)$, there exists a function $\uppsi$ that can be approximated by a sequence $\{\uppsi_n\}_{n\in\N}\subset\mathscr{E}_{\A}(\H)$ such that 
(see \cite[pp. 75]{EANK})
\begin{align}\label{frd2}
	\lim\limits_{n\to\infty}\uppsi_n=\uppsi \ \text{ and } \lim\limits_{n\to\infty}(\kappa\uppsi_n-\mathcal{N}_0\uppsi_n)= g \  \text{ in } \ \mathrm{L}^2(\H;\eta).
\end{align} 
In order to prove that the Kolmogorov operator is essentially $m-$dissipative, we need to control the integral quantity $$\int_{\H}\|\mathcal{A}\x\|_{\H}^2\eta(\d\x).$$ The following result is useful to get the convenient bound of this integral, which plays a key role in the argument.
    
    \begin{proposition}\label{lem5.2}
    	Assume that the Hypothesis \ref{trQ1} holds and let
    	\begin{align}\label{assmptg}
         \mathpzc{K}_1:=\int_{Z}\|\G(z)\|_{\V}^{2m}
    	\lambda(\d z)<+\infty,
    	\end{align} 
    	for some $m\geq1$. Then, for $r>3$ with $\alpha>\varrho$ and $r=3$ with $2\beta\mu\geq1$, we have 
    	\begin{align}\label{4p44}
    	&\int_{\H}\|\y\|_{\V}^{2m-2}\|\mathcal{A}\y\|_{\H}^2\eta(\d\y) +\int_{\H}\|\y\|_{\V}^{2m}\d\eta(\y)+
    	\int_{\H}\|\y\|_{\V}^{2m-2}\|\y\|_{\wi\L^{r+1}}^{r+1}\d\eta(\y)
    	\nonumber\\&\quad+
    	\int_{\H}\|\y\|_{\V}^{2m-2} \||\y|^{\frac{r-1}{2}}\nabla\y\|_{\H}^{2}\d\eta(\y)
        \leq C,
    	\end{align}
    	where the constant $C=\mathrm{const}\{\Tr(\Q_1),\mathpzc{K}_1, m,\kappa,\mu,\beta\}$.
    \end{proposition}
    \begin{proof}
    		Let us set 
    	\begin{align}
    		\upvarphi(\y)=\|\y\|_{\V}^{2m}, \ \text{ for all }\ \y\in\V. 
    	\end{align}
    	Then, it can be easily seen that 
    	\begin{align}
    	\mathcal{D}_{\y}\upvarphi(\y)\z&=2m\|\y\|_{\V}^{2m-2}\big((\mathcal{A}+\I)^{\frac12}\y,
    	(\mathcal{A}+\I)^{\frac12}\z\big)\label{est1}, \ \text{ for all }  \z\in\V,\\
    	\mathcal{D}_{\y}^2\upvarphi(\y)(\z_1,\z_2)&=2m(2m-2)\|\y\|_{\V}^{2m-4}
    	\big((\mathcal{A}+\I)^{\frac12}\y,
    	(\mathcal{A}+\I)^{\frac12}\z_1\big)
    	\big((\mathcal{A}+\I)^{\frac12}\y,(\mathcal{A}+\I)^{\frac12}\z_2\big)
    	\nonumber\\&\quad+
    	2m\|\y\|_{\V}^{2m-2}\big((\mathcal{A}+\I)\z_1,\z_2\big), \ \text{ for all }  \z_1,\z_2\in\V \label{est2}
    	\end{align}
    	and 
    	\begin{align}\label{qaeaq}
    		\Tr[\Q\mathcal{D}_{\y}^2\upvarphi(\y)]=2m(2m-1)\|\y\|_{\V}^{2m-2}
    		\mathrm{Tr}(\Q+\Q_1).
    	\end{align}
    	
    		Together with above expressions, we deduce the following:
    	\begin{align}\label{n2op}
    		\mathcal{N}_2\upvarphi(\y)&=\frac{1}{2}\Tr[\Q\mathcal{D}_{\y}^2\upvarphi(\y)]-(\mu\mathcal{A}\y+\alpha\y+
    		\mathfrak{B}(\y)+\beta\mathfrak{C}(\y),\mathcal{D}_{\y}\upvarphi(\y))
    		\nonumber\\&\quad+
    		\int_{Z}\left[\upvarphi(\y+\G(z))-\upvarphi(\y)-(\G(z),\mathcal{D}_{\y}\upvarphi(\y))\right]\lambda(\d z) \nonumber\\&=
    	    m(2m-1)\|\y\|_{\V}^{2m-2}\mathrm{Tr}(\Q+\Q_1)-
    	    2\mu m\|\y\|_{\V}^{2m-2}(\|\mathcal{A}\y\|_{\H}^2+\|\nabla\y\|_{\H}^2)-
    	    2\alpha m\|\y\|_{\V}^{2m}
    	    \nonumber\\&\quad-
    		2m\|\y\|_{\V}^{2m-2}(\mathfrak{B}(\y),\mathcal{A}\y)-
    		2\beta m\|\y\|_{\V}^{2m-2}\big[(\mathfrak{C}(\y),\mathcal{A}\y)+
    		\|\y\|_{\wi\L^{r+1}}^{r+1}\big]
    		\nonumber\\&\quad+
    		\int_{Z}\left[\|\y+\G(z)\|_{\V}^{2m}-\|\y\|_{\V}^{2m}-
    		2m\|\y\|_{\V}^{2m-2}(\G(z),(\mathcal{A}+\I)\y)\right]
    		\lambda(\d z).
    	\end{align}
    Due to the invariance of $\eta$, we have 
    	\begin{align}\label{pttra}
    	\int_{\H}\mathtt{P}_t\upvarphi(\y)\eta(\d\y)&=\int_{\H}\upvarphi(\y)\eta(\d\y)\nonumber\\
    	\text{ or } \int_{\H} \frac{1}{t}(\mathtt{P}_t\upvarphi(\y)-\upvarphi(\y))\eta(\d\y)&=0 \ 
    	\text{ for all } \ t>0.
    	\end{align}
    	By the definition of  $\mathcal{N}_2$ along with the Lebesgue Dominated Convergence Theorem, we may pass to the limit  in \eqref{pttra} to deduce that $$\int_{\H}\mathcal{N}_2\upvarphi(\y)\eta(\d\y)=0.$$  Therefore, we infer   from \eqref{n2op} that
    	\begin{align}\label{4p45}
    		&2\mu m \int_{\H}\|\y\|_{\V}^{2m-2}\big[\|\mathcal{A}\y\|_{\H}^2+ \|\nabla\y\|_{\H}^2\big]\eta(\d\y) +2\alpha m\int_{\H}
    		\|\y\|_{\V}^{2m}\d\eta(\y)
    		\nonumber\\&\quad+
    		2\beta  m\int_{\H}\|\y\|_{\V}^{2m-2}\big[(\mathfrak{C}(\y),\mathcal{A}\y)+
    		\|\y\|_{\wi\L^{r+1}}^{r+1}\big]\d\eta(\y)
    		\nonumber\\&=
    		m(2m-1)\mathrm{Tr}(\Q+\Q_1)\int_{\H}\|\y\|_{\V}^{2m-2}\d\eta(\y)- 
    		2m\int_{\H}\|\y\|_{\V}^{2m-2}(\mathfrak{B}(\y),\mathcal{A}\y)\d\eta(\y)
    		\nonumber\\&\quad+
    		\int_{\H}\left[\int_{Z}\left[\upvarphi(\y+\G(z))-\upvarphi(\y)-2m\|\y\|_{\V}^{2m-2}
    		(\G(z),(\mathcal{A}+\I)\y)\right]
    		\lambda(\d z)\right]\d\eta(\y)
    	\end{align}
    	Utilizing the equality \eqref{toruseq} and \eqref{syymB3} into \eqref{4p45}, we obtain
    	\begin{align}\label{4p451}
    		&\mu m \int_{\H}\|\y\|_{\V}^{2m-2}\big[\|\mathcal{A}\y\|_{\H}^2+ 2\|\nabla\y\|_{\H}^2\big]\eta(\d\y) +2\alpha m\int_{\H}
    		\|\y\|_{\V}^{2m}\d\eta(\y)
    		\nonumber\\&+
    		\frac{3\beta m}{2}\int_{\H}\|\y\|_{\V}^{2m-2} \||\y|^{\frac{r-1}{2}}\nabla\y\|_{\H}^{2}\d\eta(\y)+
    	     2\beta m\int_{\H}\|\y\|_{\V}^{2m-2}\|\y\|_{\wi\L^{r+1}}^{r+1}\d\eta(\y)
    		\nonumber\\&\leq
    	    m(2m-1)\mathrm{Tr}(\Q+\Q_1)\int_{\H}\|\y\|_{\V}^{2m-2}\d\eta(\y)+
    		2\varrho m\int_{\H}\|\y\|_{\V}^{2m-2}\|\nabla\y\|_{\H}^2\d\eta(\y)+ 
    		\nonumber\\&\quad+
    		\int_{\H}\left[\int_{Z}\left[\|\y+\G(z)\|_{\V}^{2m}-\|\y\|_{\V}^{2m}-2m\|\y\|_{\V}^{2m-2}
    		(\G(z),(\mathcal{A}+\I)\y)\right]\lambda(\d z)\right]\d\eta(\y).
    	\end{align}
    By an application of Taylor's formula with integral remainder \cite[Theorem 7.9.1, pp. 508]{PGC}, we calculate
    \begin{align}\label{4pgte1}
    &\|\y+\G(z)\|_{\V}^{2m}-\|\y\|_{\V}^{2m}-2m\|\y\|_{\V}^{2m-2}
    (\G(z),(\mathcal{A}+\I)\y)
    \nonumber\\&=\int_0^1
    \bigg[2m(2m-2)\|\y+\theta\G(z)\|_{\V}^{2m-4}\big((\mathcal{A}+\I)^{\frac12}
    (\y+\theta\G(\z)),(\mathcal{A}+\I)^{\frac12}\G(z)\big)^2
    \nonumber\\&\quad+
    2m\|\y+\theta\G(\z)\|_{\V}^{2m-2}\big((\mathcal{A}+\I)^{\frac12}\G(z),
    (\mathcal{A}+\I)^{\frac12}\G(z)\big)\bigg]\d\theta
     \nonumber\\&\leq
     2m(2m-1)\big[\|\y\|_{\V}^{2m-2}+\|\G(z)\|_{\V}^{2m-2}\big]
     \|\G(z)\|_{\V}^2
     \nonumber\\&\leq
     \frac{\alpha-\varrho}{2}\|\y\|_{\V}^{2m}+\varrho_2\|\G(z)\|_{\V}^{2m},
    \end{align}
   where $\varrho_2:=1+\left[\frac{2(2m-2)(2m-1)}{\alpha-\varrho}\right]^{m-1}\frac{1}{m}$.
   Moreover, by an application of Young's inequality, we find
   \begin{align}\label{4pgte2}
   	m(2m-1)\mathrm{Tr}(\Q+\Q_1)\|\y\|_{\V}^{2m-2}
   	\leq\frac{\alpha-\varrho}{2}\|\y\|_{\V}^{2m}+\varrho_3,
   \end{align}
   where $\varrho_3:=\frac{1}{m}\left[\frac{2(m-1)}{m(\alpha-\varrho)}\right]^{m-1}
   \left(m(2m-1)\mathrm{Tr}(\Q+\Q_1)\right)^{m}$. 
       On combining \eqref{4pgte1}-\eqref{4pgte2} and substituting into \eqref{4p451}, we obtain
   \begin{align*}
   	&\mu m \int_{\H}\|\y\|_{\V}^{2m-2}\big[\|\mathcal{A}\y\|_{\H}^2+ 2\|\nabla\y\|_{\H}^2\big]\eta(\d\y) +(\alpha-\varrho)m\int_{\H}
   	\|\y\|_{\V}^{2m}\d\eta(\y)
   	\nonumber\\&\quad+
   	\frac{3\beta m}{2}\int_{\H}\|\y\|_{\V}^{2m-2} \||\y|^{\frac{r-1}{2}}\nabla\y\|_{\H}^{2}\d\eta(\y)+
   	2\beta m\int_{\H}\|\y\|_{\V}^{2m-2}\|\y\|_{\wi\L^{r+1}}^{r+1}\d\eta(\y)
   	\nonumber\\&\leq
   	\varrho_3+\varrho_2
   	\int_{\H}\left[\int_{Z}\|\G(z)\|_{\V}^{2m}\lambda(\d z)\right]\d\eta(\y),
   \end{align*}
   which yields \eqref{4p44}. Similarly, one can also establish the above estimate for $r=3$ with $2\beta\mu\geq1$. 
    \end{proof}

    \section{Essential $m$-dissipativity of the Kolmogorov operator: Main result}\label{disscore}\setcounter{equation}{0}
    We now turn to the proof that the Kolmogorov operator $\mathcal{N}_0$ is essentially $m$-dissipative. In particular, we show that 
    $\overline{\mathcal{N}_0}^{\|\cdot\|_{\mathrm{L}^2(\H;\eta)}}=\mathcal{N}_2.$
    For completeness, we briefly review the relevant definitions and results required in what follows:
    
    We say that a linear operator $\mathscr{A}:\D(\mathscr{A})\subset\mathcal{H}\to\mathcal{H}$ in a Hilbert space $\mathcal{H}$  is \emph{dissipative} if
    \begin{align}\label{dissp}
    	\|\upvarphi\|_{\mathcal{H}}\leq \frac{1}{\kappa}\|\kappa\upvarphi-\mathscr{A}\upvarphi\|_{\mathcal{H}}\ \text{ for all }\ \upvarphi\in\D(\mathscr{A}),\ \kappa>0. 
    \end{align}
    Any dissipative operator is \emph{closable}, that is, it has a closed extension. The dissipative operator $\mathscr{A}$  is called \emph{$m$-dissipative} if $\mathrm{Range}(\kappa\I-\mathscr{A})=\mathcal{H}$ for some $\kappa>0$. An operator $\mathscr{A}$ with dense domain is $m$-dissipative if and only if it is the infinitesimal generator of a strongly continuous semigroup of contractions in $\mathcal{H}$. 
%    The following result is crucial to get the required argument.
%       \begin{theorem}[Lumer-Phillips, {\cite[Theorem 3.20]{gdp1}}]
%    	Let $\mathscr{A}:\D(\mathscr{A})\subset\mathcal{H}\to\mathcal{H}$  be a dissipative operator in the Hilbert space $\mathcal{H}$ such that $\D(\mathscr{A})$ is dense in $\mathcal{H}$.  Assume that for some $\kappa>0$,  the range of $\kappa\I-\mathscr{A}$ is dense in $\mathcal{H}$. Then the closure of $\mathscr{A}$ is $m$-dissipative. 
%    \end{theorem}

    \begin{proposition}[$\mathcal{N}_2$ is an extension of $\mathcal{N}_0$]\label{n2en0}
    Assume that the assumptions of Proposition \ref{lem5.2} hold. Then, for any $\upvarphi\in\mathscr{E}_{\mathcal{A}}(\H)$, we have
    	\begin{align*}
    		\upvarphi\in\D(\mathcal{N}_2)  \ \text{ and } \ \mathcal{N}_2\upvarphi
    		=\mathcal{N}_0\upvarphi.
    	\end{align*}
    \end{proposition}
    \begin{proof}
    	Let $\upvarphi\in\mathscr{E}_{\mathcal{A}}(\H)$. On employing the infinite-dimensional  It\^o formula \cite[Theorem 3.7.2, Chapter 3]{VMBR} to the function $\upvarphi(\cdot)$ and to the process $\Y(\cdot)=\Y(\cdot,\y)$, we have $\mathbb{P}-$a.s.
    \begin{align}\label{450}
    	&\upvarphi(\Y(t))-\upvarphi(\y)
    	\nonumber\\&=-
    	\int_0^t\big(\mu \mathcal{A}\Y(s)+\alpha\Y(s)+\mathfrak{B}(\Y(s))+\beta\mathfrak{C}(\Y(s)),
    	\mathcal{D}_{\y}\upvarphi(\Y(s))\big)\d s+ \frac{1}{2}\int_0^t\Tr(\Q\D^2_{\y}\upvarphi(\Y(s)))\d s 
    	\nonumber\\&\quad+\underbrace{
    		\int_0^t \big(\sqrt{\Q}\d\W(s),\mathcal{D}_{\y}\upvarphi(\Y(s))\big)+
    		\int_0^t\int_{Z}\big[\upvarphi(\Y(s-)+\G(z))-\upvarphi(\Y(s-))\big]
    		\wi{\pi}(\d s,\d z)}_{\mathscr{F}_t-\text{ adapted martingale }}
    	\nonumber\\&\quad+
    	\int_0^t\int_{Z}\big[\upvarphi(\Y(s-)+\G(z))-\upvarphi(\Y(s-))-
    	(\mathcal{D}_{\y}\upvarphi(\Y(s-)),\G(z))\big]\lambda(\d z)\d s,
    \end{align}
    for all $t\in[0,T]$. Consequently on taking the expectation in \eqref{450}, we find
    \begin{align}\label{451}
    	\mathtt{P}_t\upvarphi(\y)&=\E\left[\upvarphi(\Y(t,\y))\right]\nonumber\\&=\upvarphi(\y)-
    	\E\left[\int_0^t\big(\mu \mathcal{A}\Y(s)+\alpha\Y(s)+\mathfrak{B}(\Y(s))+\beta\mathfrak{C}(\Y(s)),
    	\mathcal{D}_{\y}\upvarphi(\Y(s))\big)\d s\right] 
    	\nonumber\\&\quad+
    	\frac{1}{2}\E\left[\int_0^t\Tr(\Q\D^2_{\y}\upvarphi(\Y(s)))\d s\right]
    	\nonumber\\&\quad+
    	\E\left[	\int_0^t\int_{Z}\big[\upvarphi(\Y(s-)+\G(z))-\upvarphi(\Y(s-))-
    	(\mathcal{D}_{\y}\upvarphi(\Y(s-)),\G(z))\big]\lambda(\d z)\d s\right]
    	\nonumber\\&=
    	\upvarphi(\y)+\E\left[\int_0^t\mathcal{N}_0\upvarphi(\Y(s,\y))\d s\right],
    \end{align}
    for all $t\in[0,T]$. Therefore, by the definition of $\mathcal{N}_2$, we infer
    \begin{align}
    	\mathcal{N}_2\upvarphi(\y)=\lim\limits_{t\to 0^+}\frac{1}{t}\left(\mathtt{P}_t\upvarphi(\y)-\upvarphi(\y)\right)=\lim\limits_{t\to 0^+}\frac{1}{t}\E\left[\int_0^t \mathcal{N}_0\upvarphi(\Y(s,\y))\d s\right]= \mathcal{N}_0\upvarphi(\y),\ \y\in\H,
    \end{align}
    pointwise. We will now demonstrate that $\upvarphi\in\D(\mathcal{N}_2)$. It is enough to show that
    \begin{align}
    	\frac{1}{t}(\mathrm{P}_t\upvarphi-\upvarphi)
    \end{align}
    is equibounded in $\mathrm{L}^2(\H;\eta)$ for $t\in(0,1]$. Note that for any $\upvarphi\in\mathscr{E}_{\A}(\H),$ there exist constants $C_1$ and $C_2$ depending on $\upvarphi$ such that 
    \begin{align}\label{kolm}
    	\left|\frac{1}{2}\Tr\left[\Q\mathcal{D}_{\y}^2\upvarphi(\y)\right]-(\alpha\y,\mathcal{D}_{\y}\upvarphi(\y))\right|\leq C_1+C_2\|\y\|_{\H}, \ \y\in\H,
    \end{align}
    see, for instance, \cite[Chapter 2, pp. 38]{gdp1}. Moreover, an application of Taylor's formula yields 
    \begin{align*}
    	&\left|\int_{Z}\left[\upvarphi(\y+\G(z))-\upvarphi(\y)-(\G(z),\mathcal{D}_{\y}\upvarphi(\y))\right]\lambda(\d z)\right|\nonumber\\&=\left|\int_{Z}\int_0^1(1-\theta)\mathcal{D}_{\y}^2\upvarphi(\y+\theta\G(z))(\G(z)\otimes\G(z))\d\theta\lambda(\d z)\right|\nonumber\\&\leq C_3\int_{Z}\|\G(z)\|_{\H}^2\lambda(\d z),\ \y\in\H. 
    \end{align*}
    In view of \eqref{451} and \eqref{kolm}, we have in fact 
    \begin{align}
    	|\mathrm{P}_t\upvarphi(\y)-\upvarphi(\y)|&=
    	\left|\E\left[\int_0^t \mathcal{N}_0\upvarphi(\Y(s,\y))\d s\right]\right| \nonumber\\&\leq
    	\int_0^t\E\bigg[C_1+C_2\|\Y(s,\y)\|_{\H}+\|\upvarphi\|_1\big(\mu\|\mathcal{A}\Y(s,\y)\|_{\H}+\|\mathfrak{B}(\Y(s,\y))\|_{\H}+
    	\nonumber\\&\qquad\qquad+
    	\alpha\|\Y(s,x)\|_{\H}+\beta\|\mathfrak{C}(\Y(s,\y))\|_{\H}\big)+
    	C_3\int_{Z}\|\G(z)\|_{\H}^2\lambda(\d z)\bigg]\d s,
    \end{align}
    for any $\y\in\H$. An application of H\"older's inequality yields 
    \begin{align}
    	|\mathrm{P}_t\upvarphi(\y)-\upvarphi(\y)|^2
        &\leq
        t\int_0^t\bigg\{\E\bigg[C_1+C_2\|\Y(s,\y)\|_{\H}+\|\upvarphi\|_1
    	\big(\mu\|\mathcal{A}\Y(s,\y)\|_{\H}+\|\mathfrak{B}(\Y(s,\y))\|_{\H}
    	\nonumber\\&\qquad\qquad
    	+\alpha\|\Y(s,\y)\|_{\H}+\beta\|\mathfrak{C}(\Y(s,\y))\|_{\H}\big)
    	+C_3\int_{Z}\|\G(z)\|_{\H}^2\lambda(\d z)\bigg]\bigg\}^2\d s 
    	\nonumber\\&\leq  Ct\int_0^t\bigg[C_1^2+C_2^2\E\|\Y(s,\y)\|_{\H}^2+\|\upvarphi\|_{1}^2
    	\big(\mu^2\E\|\mathcal{A}\Y(s,\y)\|_{\H}^2+\E\|\mathfrak{B}(\Y(s,\y))\|_{\H}^2
    	\nonumber\\&\qquad\qquad+
    	\alpha^2\E\|\Y(s,\y)\|_{\H}^2+\beta^2\E\|\mathfrak{C}(\Y(s,\y))\|_{\H}^2\big)
    	+C_3^2\left(\int_{Z}\|\G(z)\|_{\H}^2\lambda(\d z)\right)^2\bigg]\d s 
    	\nonumber\\&=  Ct\int_0^t\bigg[C_1^2+C_2^2\mathrm{P}_s(\|\cdot\|_{\H}^2(\y))+\|\upvarphi\|_{1}^2\big(\mu^2\mathrm{P}_s(\|\mathcal{A}\cdot\|_{\H}^2(\y))+\mathrm{P}_s(\|\mathfrak{B}(\cdot)\|_{\H}^2(\y))
    	\nonumber\\&\qquad\qquad+
    	\alpha^2\mathrm{P}_s(\|\cdot\|_{\H}^2(\y))+
    	\beta^2\mathrm{P}_s(\|\mathfrak{C}(\cdot)\|_{\H}^2(\y))\big)
    	+C_3^2\left(\int_{Z}\|\G(z)\|_{\H}^2\lambda(\d z)\right)^2\bigg]\d s.
    \end{align}
    Integrating with respect to $\eta$ over $\H$ and taking into account the invariance of $\eta$ along with \eqref{assmptg}, we derive 
    \begin{align}\label{traq}
    	\|\mathrm{P}_t\upvarphi-\upvarphi\|_{\mathrm{L}^2(\H;\eta)}^2
    	&\leq Ct^2\int_{\H}\bigg[C_1^2+C_2^2\|\y\|_{\H}^2+\mu^2\|\mathcal{A}\y\|_{\H}^2+
    	\|\mathfrak{B}(\y)\|_{\H}^2+\alpha^2\|\y\|_{\H}^2
    	\nonumber\\&\qquad\quad+
    	\beta^2\|\mathfrak{C}(\y)\|_{\H}^2+
    	C_3^2\left(\int_{Z}\|\G(z)\|_{\H}^2\lambda(\d z)\right)^2\bigg]\d\eta<+\infty, 
    \end{align}
    by using \eqref{4p44}. Therefore, from Proposition \ref{lem5.2}, it is clear that $\mathcal{N}_2\upvarphi\in\mathrm{L}^2(\H;\eta)$, and  by \eqref{451} and the Lebesgue dominated convergence theorem, one can deduce 
    \begin{align}
    	\mathcal{N}_2\upvarphi=\lim\limits_{t\to 0^+} \frac{1}{t}\left(\mathtt{P}_t\upvarphi-\upvarphi\right)=\mathcal{N}_0\upvarphi \ \text{ exists in }\ \mathrm{L}^2(\H;\eta). 
    \end{align}
    Therefore, $\mathcal{N}_2$ extends $\mathcal{N}_0$. 
    \end{proof}

     \subsection{An approximated problem}
    To prove that $\mathcal{N}_2$ is the closure of $\mathcal{N}_0$, it is convenient to approximate the system \eqref{scbfabs} by the following regular system: 
    \begin{equation}\label{scbfabsaprx}
    	\left\{
    	\begin{aligned}
    		\d\Y_{\e}(t)&+[\mu \mathcal{A}\Y_{\e}(t)+\alpha\Y_{\e}(t)+\mathfrak{B}_{\e}(\Y_{\e}(t))+
    		\alpha\X_{\e}(t)+\beta\mathfrak{C}_{\e}(\Y_{\e}(t))]\d t \\&=\sqrt{\Q}\d\W(t)+\int_{Z}\G(z)\wi{\pi}(\d t,\d z), \ t\in(0,T),\\
    		\Y_{\e}(0)&=\y\in\H,
    	\end{aligned}
    	\right.
    \end{equation}
    where 
    \begin{align}\label{modiB}
    	\mathfrak{B}_{\e}(\y)=\left\{\begin{array}{cc}\mathfrak{B}(\y)&\text{ if }\ \|\y\|_{\V}\leq\e^{-1},\\
    		\e^{-2}\|\y\|_{\V}^{-2}\mathfrak{B}(\y)&\text{ if }\ \|\y\|_{\V}>\e^{-1},\end{array}\right.
    \end{align}
    and 
    \begin{align}\label{modiC}
    	\mathfrak{C}_{\e}(\y)=\left\{\begin{array}{cc}\mathfrak{C}(\y)&\text{ if }\ \|\y\|_{\V}\leq\e^{-1},\\
    		\e^{-(r+1)}\|\y\|_{\V}^{-(r+1)}\mathfrak{C}(\y)&\text{ if }\ \|\y\|_{\V}>\e^{-1}.\end{array}\right.
    \end{align}
    For each$\e>0$, the operators $\mathfrak{B}_{\e}(\cdot)$ and $\mathfrak{C}_{\e}(\cdot)$ are smooth and bounded. Consequently, the system \eqref{scbfabsaprx} admits a \emph{unique strong solution} in the sense of Definition \ref{defjumpe}, which we denote by $\Y_{\e}(t,\y)$. 
    
    Associated with \eqref{scbfabsaprx}, let $\{\mathtt{P}_t^{\e}\}_{t\geq0}$ be the corresponding transition semigroup, defined for $\uppsi\in\C_b(\H)$ by
    \begin{align*}
    	\mathtt{P}_t^\e\uppsi(\y)=\E[\uppsi(\Y_{\e}(t,\y))].
    \end{align*}  
    Moreover, the bound in \eqref{tight} remains valid for the approximating system \eqref{scbfabsaprx}. This property guarantees the existence of an invariant measure $\eta_{\e}$ for the semigroup $\{\mathtt{P}_t^{\e}\}_{t\geq 0}$, and the family $\{\eta_{\e}\}_{\e>0}$ is tight. 
    Finally, since the  limiting semigroup $\{\mathtt{P}_t\}_{t\geq0}$ admits a unique invariant measure $\eta$, it follows that $\eta$ is the weak limit of the sequence $\{\eta_{\e}\}_{\e>0}$. 
    
    Equipped with preceding lemmas, we can now establish our main result.
    
    \begin{theorem}[$\mathcal{N}_2$ is the closure of $\mathcal{N}_0$]\label{thm4.7}
    Suppose that the conditions of Proposition \ref{lem5.2} are satisfied for $m=2,3$ and $m=\frac{r+1}{2}$. Then, $\mathcal{N}_0$ is dissipative in $\mathrm{L}^2(\H;\eta)$, and its closure $\overline{\mathcal{N}_0}$ in $\mathrm{L}^2(\H;\eta)$ coincides with the infinitesimal generator $\mathcal{N}_2$ of $\{\mathtt{P}_t\}_{t\geq 0}$ in $\mathrm{L}^2(\H;\eta)$. 
    \end{theorem}
    \begin{proof}
    	The proof of the theorem is divided into the following steps:
    	\vskip 0.1 cm
    	\noindent
    	\textbf{Step 1:} \emph{$\mathcal{N}_0$ is closable.} 
    	Since $\mathcal{N}_2$ is dissipative (see \eqref{453} below), therefore, by definition \eqref{dissp}, we have
    	\begin{align*}
    		\|\upvarphi\|_{\H}\leq\frac{1}{\kappa}\|\kappa\upvarphi-\mathcal{N}_2\upvarphi\|_{\H} \ \text{ for all }\  \upvarphi\in\D(\mathcal{N}_2).
    	\end{align*}
    	Using the fact that $\mathcal{N}_2$ is an extension $\mathcal{N}_0$ (see Proposition \ref{n2en0}), the above inequality implies that
    	\begin{align*}
    		\|\upvarphi\|_{\H}\leq\frac{1}{\kappa}\|\kappa\upvarphi-\mathcal{N}_0\upvarphi\|_{\H}\  \text{ for all }\  \upvarphi\in\mathscr{E}_{\A}(\H),
    	\end{align*}
    	and hence $\mathcal{N}_0$ is dissipative. Moreover, since every disspative operator is closable, so is $\mathcal{N}_0$. 
    	\vskip 0.1 cm
    	\noindent
    	\textbf{Step 2:} \emph{Closure of $\mathcal{N}_0$ coincides with $\mathcal{N}_2$.} 
    	Since $\mathcal{N}_0$ is closable, it follows from \cite[Chapter 1, Lemma 1.8]{EBD} that it admits a \emph{minimal closed extension} namely its \emph{closure}, which we denote by $\overline{\mathcal{N}_0}$.
    	Our objective is to prove that $\overline{\mathcal{N}_0}=\mathcal{N}_2$. To this end, we analyze the approximating system \eqref{scbfabsaprx} together with the invariant measure $\eta$  associated with the semigroup $\{\mathtt{P}_t\}_{t\geq 0}$.
    	
    	Let $f\in\C_b^1(\H)$. For $\upvarphi\in\mathscr{E}_{\A}(\H)$, we can express 
    	\begin{align*}
    		(\mathcal{N}_{\e}\upvarphi)(\y)
    		&=\frac{1}{2}\Tr\left[\Q\mathcal{D}_{\y}^2\upvarphi(\y)\right]-(\mu\mathcal{A}\y+
    		\alpha\y+\mathfrak{B}_{\e}(\y)+\beta\mathfrak{C}_{\e}(\y),\mathcal{D}_{\y}\upvarphi(\y))
    		\nonumber\\&\quad+
    		\int_{Z}\left[\upvarphi(\y+\G(z))-\upvarphi(\y)-(\G(z),\mathcal{D}_{\y}\upvarphi(\y))\right]
    		\lambda(\d z),
    	\end{align*}
    	which can be written as
    	\begin{align*}
    		(\mathcal{N}_{\e}\upvarphi)(\y)=
    	(\mathcal{N}_0\upvarphi)(\y)-(\mathfrak{B}_{\e}(\y)-\mathfrak{B}(\y),\mathcal{D}_{\y}\upvarphi(\y))-\beta(\mathfrak{C}_{\e}(\y)-\mathfrak{C}(\y),\mathcal{D}_{\y}\upvarphi(\y)).
    	\end{align*}
    	Since $\mathfrak{B}_{\e}(\cdot)$ and $\mathfrak{C}_\e(\cdot)$ are bounded and regular, the resolvent equation
    	\begin{align*}
    		\kappa\upvarphi_{\e}-\mathcal{N}_{\e}\upvarphi_{\e}=f 
    	\end{align*}
    	admits a unique solution  $\upvarphi_{\e}\in\D(\mathcal{N}_{\e})\cap\C_b^1(\H)$. This solution is given by the resolvent representation (see \cite[Theorem 3.21, Step 1]{gdp1})
    	\begin{align*}
    		\upvarphi_{\e}(\y)=\int_0^{\infty}e^{-\kappa t}\E\left[f(\Y_{\e}(t,\y))\right]\d t,
    	\end{align*}
    	where $\Y_\e(\cdot)$ denotes the solution of \eqref{scbfabsaprx}.
    	
    	In particular, $\upvarphi_{\e}$ is differentiable along any direction $\boldsymbol{h}\in\H$, and its directional derivative is given by
    	\begin{align*}
    		(\mathcal{D}_{\y}\upvarphi_{\e}(\y),\boldsymbol{h})=\int_0^{\infty}e^{-\kappa t}\E\left[(\mathcal{D}_{\y}f(\Y_{\e}(t,\y)),\boldsymbol{\xi}_{\e}^{\boldsymbol{h}}(t,\y))\right]\d t,
    	\end{align*}
    	where $\boldsymbol{\xi}_{\e}^h(t,\y)=
    		\mathcal{D}_{\y}\Y_{\e}(t,\y)\boldsymbol{h}$.
    	By using the Cauchy-Schwarz and H\"older inequalities, and Lemma \ref{lem4.4} (see \eqref{4p20}), it follows that 
    	\begin{align*}
    		|(\mathcal{D}_{\y}\upvarphi_{\e}(\y),\boldsymbol{h})|&\leq \int_0^{\infty}e^{-\kappa t}\left\{\E\left[\|\mathcal{D}_{\y}f(\Y_{\e}(t,\y))\|_{\H}^2\right]\right\}^{1/2}\E\left\{\left[\|\boldsymbol{\xi}_{\e}^{\boldsymbol{h}}(t,\y)\|_{\H}^2\right]\right\}^{1/2}\d t\nonumber\\&\leq \frac{1}{\kappa+\alpha-\varrho_1}\|\boldsymbol{h}\|_{\H}\|f\|_1, 
    	\end{align*}
    	 for all $\boldsymbol{h},\y\in\H.$ Consequently, the arbitrariness of $\boldsymbol{h}\in\H$ yields
    	\begin{align}\label{bism}
    		\|\mathcal{D}_{\y}\upvarphi_{\e}(\y)\|_{\H}\leq \frac{1}{\kappa+\alpha-\varrho_1}
    		\|f\|_1, \  \text{ for all }\ \y\in\H. 
    	\end{align}
    	Proceeding along the same lines as in \cite[Claim 1, pg. 171]{gdp1}, one deduces that $\upvarphi_{\e}\in\D(\overline{\mathcal{N}_0})$ and satisfies
    	\begin{align}\label{bece}
    		\lambda\upvarphi_{\e}-\overline{\mathcal{N}_0}\upvarphi_{\e}=
    		(\mathfrak{B}_{\e}(\y)-\mathfrak{B}(\y),\mathcal{D}_{\y}\upvarphi_{\e}(\y))+
    		\beta(\mathfrak{C}_{\e}(\y)-\mathfrak{C}(\y),\mathcal{D}_{\y}\upvarphi_{\e}(\y))+f. 
    	\end{align}
    	Our next objective is to prove that the perturbation terms vanish in the limit, namely 
    	\begin{align}\label{457}
    		\lim\limits_{\e\to 0}(\mathfrak{B}_{\e}(\y)-\mathfrak{B}(\y), \mathcal{D}_{\y}\upvarphi_{\e}(\y))=0\ \text{ in }\ \mathrm{L}^2(\H;\eta)
    	\end{align}
    	and 
    	\begin{align}\label{4p57}
    		\lim\limits_{\e\to 0}(\mathfrak{C}_{\e}(\y)-\mathfrak{C}(\y),\mathcal{D}_{\y}\upvarphi_{\e}(\y))=0\ \text{ in }\ \mathrm{L}^2(\H;\eta). 
    	\end{align}
    	By using the Cauchy-Schwarz inequality, \eqref{modiB} and \eqref{bism}, we note that
    	\begin{align*}
    		&\int_{\H}\left|(\mathfrak{B}_{\e}(\y)-\mathfrak{B}(\y),\mathcal{D}_{\y}\upvarphi_{\e}(\y))\right|^2\eta(\d\y)
    		\nonumber\\&=
    		\int_{\{\|\y\|_{\V}\geq\e^{-1}\}}\left|(\mathfrak{B}_{\e}(\y)-\mathfrak{B}(\y),\mathcal{D}_{\y}\upvarphi_{\e}(\y))\right|^2\eta(\d\y)
    		\nonumber\\&\leq
    		\int_{\{\|\y\|_{\V}\geq\e^{-1}\}}\left|\frac{1-\e^2\|\y\|_{\V}^2}{\e^2\|\y\|_{\V}^2}\right|\|\mathfrak{B}(\y)\|_{\H}^2\|\mathcal{D}_{\y}\upvarphi_{\e}(\y)\|_{\H}^2\eta(\d\y)
    		\nonumber\\&\leq 
    		\frac{\|f\|_1^2}{\kappa+\alpha-\varrho_1}
    		\int_{\{\|\y\|_{\V}\geq\e^{-1}\}}
    		\|\mathfrak{B}(\y)\|_{\H}^2\eta(\d\y),
    	\end{align*}
    	so that \eqref{457} follows if 
    	\begin{align}\label{4p59}
    		\lim\limits_{\e\to 0}\int_{\{\|\y\|_{\V}\geq\e^{-1}\}}\|\mathfrak{B}(\y)\|_{\H}^2\eta(\d\y)=0. 
    	\end{align}
       It is enough to show that 
    	\begin{align}\label{Best4}
    	\int_{\H}\|\mathfrak{B}(\y)\|_{\H}^2\eta(\d\y)<+\infty.
    	\end{align}
    	Indeed, for fixed $\y\in\V$, we have $\|\y\|_{\V}<+\infty$ (does not depend on $\e$). Therefore, for sufficiently small $\e>0$, the indicator function 
    	$$\mathds{1}_{\{\|\y\|_{\V}\geq\e^{-1}\}}\to0
    	\ \text{ as } \ \e\to0.$$
    	In particular, we have
    	$$\|\mathfrak{B}(\y)\|_{\H}^2\mathds{1}_{\{\|\y\|_{\V}\geq\e^{-1}\}}\to0
    	\ \text{ as } \ \e\to0.$$
    	Moreover, since $\|\mathfrak{B}(\y)\|_{\H}^2\mathds{1}_{\{\|\y\|_{\V}\geq\e^{-1}\}}\leq
    	\|\mathfrak{B}(\y)\|_{\H}^2$ and from \eqref{Best4}, $\|\mathfrak{B}(\y)\|_{\H}^2\in\mathrm{L}^1(\H;\eta)$, therefore, by an  application of the Lebesgue dominated convergence theorem, we finally deduce \eqref{4p59}.
    	
    	We now proceed to prove \eqref{Best4}. By applying H\"older's, Agmon's and Young's inequalities, and using  Proposition \ref{lem5.2} (see \eqref{4p44}), we deduce 
    	\begin{align*}
    	\int_{\H}\|\mathfrak{B}(\y)\|_{\H}^2\eta(\d x)&\leq
    	\int_{\H}\|\y\|_{\V}^2\|\y\|_{\wi\L^{\infty}}^2\eta(\d x)\leq \int_{\H}\|\y\|_{\V}^2
    	\|\y\|_{\H}^{2-\frac{d}{2}}\|(\mathcal{A}+\I)\y\|_{\H}^{\frac{d}{2}}\eta(\d x)
    	\nonumber\\&\leq
    	\int_{\H}\|\y\|_{\V}^2\left[\|\y\|_{\H}^2+\varrho_3\|(\mathcal{A}+\I)\y\|^2\right]
    	\eta(\d x)
    	\nonumber\\&<+\infty,
    	\end{align*}
    	where $\varrho_3:=\frac{d}{4}\left(\frac{4-d}{4}\right)^{\frac{4-d}{d}}$.
    	 By using \eqref{modiC} and \eqref{bism}, we calculate 
    	\begin{align*}
    	&\int_{\H}\left|(\mathfrak{C}_{\e}(\y)-\mathfrak{C}(\y),
    	\mathcal{D}_{\y}\upvarphi_{\e}(\y))\right|^2\eta(\d\y)
    	\nonumber\\&=
    	\int_{\{\|\y\|_{\V}\geq\e^{-1}\}}\left|(\mathfrak{C}_{\e}(\y)-\mathfrak{C}(\y),\mathcal{D}_{\y}\upvarphi_{\e}(\y))\right|^2\eta(\d\y)
    	\nonumber\\&\leq
    	\int_{\{\|\y\|_{\V}\geq\e^{-1}\}}\left|\frac{1-\e^{r+1}\|\y\|_{\V}^{r+1}}{\e^{r+1}\|\y\|_{\V}^{r+1}}\right|\|\mathfrak{C}(\y)\|_{\H}^2\|\mathcal{D}_{\y}\upvarphi_{\e}(\y)\|_{\H}^2\eta(\d\y)
    	\nonumber\\&\leq 
    	\frac{\|f\|_1^2}{\lambda+\alpha-\varrho_1}\int_{\{\|\y\|_{\V}\geq\e^{-1}\}} \|\mathfrak{C}(\y)\|_{\H}^2\eta(\d\y),
    	\end{align*}
    	so that \eqref{4p57} follows if 
    	\begin{align}\label{4pp57}
    		\lim\limits_{\e\to 0}\int_{\{\|\y\|_{\V}\geq\e^{-1}\}} \|\mathfrak{C}(\y)\|_{\H}^2 \eta(\d\y)=0. 
    	\end{align}
    	Similar to the above discussion, we only need to establish 
    \begin{align*}
    	\int_{\H}\|\mathfrak{C}(\y)\|_{\H}^2\eta(\d\y)<+\infty.
    \end{align*}
    	By using the interpolation inequality, we calculate
    	\begin{align}\label{Cest1}
    	\int_{\H}\|\mathfrak{C}(\y)\|_{\H}^2\eta(\d\y)\leq
    	\int_{\H}\|\y\|_{\wi\L^{2r}}^{2r}\eta(\d\y)
    	\leq
    	\int_{\H}\|\y\|_{\wi\L^{3(r+1)}}^{\frac{3(r-1)}{2}}
    	\|\y\|_{\wi\L^{r+1}}^{\frac{r+3}{2}}\eta(\d\y).
    	\end{align}
    	To proceed further, we consider the following cases:
    	\vskip 0.2cm
    	\noindent
    	\textbf{Case I:} \emph{When $d=3$ and $3< r<5$ and $r=3$ with $2\beta\mu\geq1$.}
    	We use the following exponent trick in $\|\cdot\|_{\wi\L^{r+1}}-$norm 
    	\begin{align}\label{Cest2}
    		\|\cdot\|_{\wi\L^{r+1}}=\|\cdot\|_{\wi\L^{r+1}}^{\frac{3(r-1)^2}{(r+1)(r+3)}+
    		\frac{2r(5-r)}{(r+1)(r+3)}},
    	\end{align}
    	where we have used the fact that $\frac{3(r-1)^2}{(r+1)(r+3)}<1$ for $r<5$. By substituting \eqref{Cest2} into \eqref{Cest1} and subsequently applying Holder's and Young's inequalities, we obtain the following estimate
    	\begin{align*}
    	\int_{\H}\|\mathfrak{C}(\y)\|_{\H}^2\eta(\d\y)&\leq
    	\int_{\H}\|\y\|_{\wi\L^{3(r+1)}}^{\frac{3(r-1)}{2}}
    	\|\y\|_{\wi\L^{r+1}}^{\frac{3(r-1)^2}{2(r+1)}}
    	\|\y\|_{\wi\L^{r+1}}^{\frac{r(5-r)}{r+1}}\eta(\d\y)
    	\nonumber\\&\leq
    	\left(\int_{\H}\|\y\|_{\wi\L^{3(r+1)}}^{r+1}
    	\|\y\|_{\wi\L^{r+1}}^{r-1}\eta(\d\y)\right)^{\frac{3(r-1)}{2(r+1)}}
    	\left(\int_{\H}\|\y\|_{\wi\L^{r+1}}^{2r}
    	\eta(\d\y)\right)^{\frac{5-r}{2(r+1)}}
    	\nonumber\\&\leq C_e\underbrace{
    	\left(\int_{\H}\|\y\|_{\wi\L^{3(r+1)}}^{r+1}
    	\|\y\|_{\V}^{r-1}\eta(\d\y)\right)^{\frac{3(r-1)}{2(r+1)}}
    	\left(\int_{\H}\|\y\|_{\wi\L^{r+1}}^{r+1}\|\y\|_{\V}^{r-1}
    	\eta(\d\y)\right)^{\frac{5-r}{2(r+1)}}}_{<+\infty\text{ from } \eqref{4p44} 
    	\text{ for } m=\frac{r+1}{2}},
    	\end{align*}
    	where $C_e$ is the constant appearing in the Sobolev embedding 
    	$\V\hookrightarrow\wi\L^{r+1}$.
    	\vskip 0.2cm
    	\noindent
    	\textbf{Case II:} \emph{When $d=3$ and $r=5$.}
    	By using the interpolation estimate, we obtain the following estimate
    	\begin{align*}
    	\int_{\H}\|\mathfrak{C}(\y)\|_{\H}^2\eta(\d\y)\leq
    	\int_{\H}\|\y\|_{\wi\L^{10}}^{10}\eta(\d\y)
    	&\leq
    	\int_{\H}\|\y\|_{\wi\L^{18}}^{6}\|\y\|_{\wi\L^{6}}^{4}\eta(\d\y)
    \leq 
    	C_e\underbrace{
    	\int_{\H}\|\y\|_{\wi\L^{18}}^{6}\|\y\|_{\V}^{4}\eta(\d\y)}_{<+\infty\text{ from } \eqref{4p44} \text{ for } m=3},
    	\end{align*}
    	where $C_e$ is the constant appearing in the Sobolev embedding 
    	$\V\hookrightarrow\wi\L^{6}$.
    	\vskip 0.2cm
    	\noindent
    	\textbf{Case III:} \emph{When $d=2$ and $r>3$, and $r=3$ with $2\beta\mu\geq1$.}
    	Let us first calculate $\|\cdot\|_{\wi\L^{2r}}-$norm by using the Gagliardo-Nirenberg Sobolev inequality as
    	\begin{align}\label{Cest3}
    		\|\y\|_{\wi\L^{2r}}\leq C_e\|\nabla\y\|_{\H}^{\frac{r-1}{2r}}
    		\|\y\|_{\wi\L^{r+1}}^{\frac{r+1}{2r}}.
    	\end{align}
    	Together with \eqref{Cest3},  we calculate
    	\begin{align*}
    		\int_{\H}\|\mathfrak{C}(\y)\|_{\H}^2\eta(\d\y)\leq
    		\int_{\H}\|\y\|_{\wi\L^{2r}}^{2r}\eta(\d\y)
    		\leq
    		C_e\underbrace{\int_{\H}\|\y\|_{\H}^{r-1}
    		\|\y\|_{\wi\L^{r+1}}^{r+1}\eta(\d\y)}_{<+\infty\text{ from } \eqref{4p44} 
    		\text{ for } m=\frac{r+1}{2}}.
    	\end{align*}
    	Consequently, it follows from \eqref{bece} that
    	\begin{align*}
    		\kappa\upvarphi_{\e}-\overline{\mathcal{N}_0}\upvarphi_{\e}\to f \ \text{ as }\ \e\to 0\ \text{ in }\ \mathrm{L}^2(\H;\eta).
    	\end{align*}
    This implies that range of $\kappa\I-\overline{\mathcal{N}_0}$ is dense in $\mathrm{L}^2(\H;\eta)$. Therefore, by the Lumer-Phillips Theorem (see \cite[Theorem 3.20]{gdp1}), the operator $\overline{\mathcal{N}_0}$ is $m$-dissipative. 
    
    On the other hand, $\mathcal{N}_2$ is also $m$-dissipative, being the infinitesimal generator of a strongly continuous contraction semigroup, and it extends $\mathcal{N}_0$. Hence, by uniqueness of $m-$dissipative extensions, we conclude that  $\mathcal{N}_2=\overline{\mathcal{N}_0}$ as desired. 
    \end{proof}

   \subsection{Some consequences of Theorem \ref{thm4.7}}\label{consekom}
   Having established the essential $m$-dissipativity of the Kolmogorov operator $\mathcal{N}_0$, we now turn to some of its important consequences, which play a crucial role in the subsequent analysis and applications.
   
    \subsubsection{The ``Carre du Champ's" identity} 
  The first important consequence of Theorem \ref{thm4.7} is the integration by parts formula, also referred to as the \emph{identit\'e du carr\'e du champ} associated with the Kolmogorov operator $\mathcal{N}_0$. 
To the best of our knowledge, such an identity has not been previously established in the existing literature of integro-differential Kolmogorov operator $\mathcal{N}_0$.
For reader's convenience, we establish here the following straightforward identity:
    \begin{align}\label{carredujump}
    	\mathcal{N}_0(\upvarphi^2)=2\upvarphi \mathcal{N}_0\upvarphi+\|\sqrt{\Q}\mathcal{D}_{\y}\upvarphi\|_{\H}^2
    	+\int_{Z}\left(\upvarphi(\cdot+\G(z))-\upvarphi(\cdot)\right)^2\lambda(\d z)
    	 \ \text{ for all } \upvarphi\in\mathscr{E}_{\A}(\H). 
    \end{align} 
    For $\upvarphi_{\h}(\y)=e^{i(\y,\h)}$, we have 
    \begin{align*}
    	\mathcal{D}_{\y}\upvarphi_{\h}= ie^{i(\y,\h)}\h, \ \mathcal{D}_{\y}^2\upvarphi_{\h}=-e^{i(\y,\h)}(\h\otimes\h), \ \y,\h\in\H.
    \end{align*}
    Then, we find
    \begin{align*}
    \mathcal{N}_0(\upvarphi_{\h}^2)(\y)&=
    -2e^{2i(\y,\h)}(\Q\h,\h)-2ie^{2i(\y,\h)}(\mu\mathcal{A}\y+\mathfrak{B}(\y)+\alpha\y+\beta\mathfrak{C}(\y),\h)
    \nonumber\\&\quad+
    e^{2i(\y,\h)}\int_{Z}\big[e^{2i(\G(z),\h)}-1-2i(\G(z),\h)\big]\lambda(\d z),\\
    \upvarphi_{\h}(\y)\mathcal{N}_0(\upvarphi_{\h})(\y)&=
    -\frac12e^{2i(\y,\h)}(\Q\h,\h)-ie^{2i(\y,\h)}(\mu\mathcal{A}\y+\mathfrak{B}(\y)+
    \alpha\y+\beta\mathfrak{C}(\y),\h)
    \nonumber\\&\quad+
    e^{2i(\y,\h)}\int_{Z}\big[e^{i(\G(z),\h)}-1-i(\G(z),\h)\big]\lambda(\d z)\\
    \text{ and } \ 
    \|\Q^{\frac12}\mathcal{D}_{\y}\upvarphi_{\h}(\y)\|_{\H}^2&=-e^{2i(\y,\h)}(\Q\h,\h).
    \end{align*}
    On combining the above identities, we deduce 
    \begin{align*}
    &\mathcal{N}_0(\upvarphi_{\h}^2)(\y)-2\upvarphi_{\h}(\y)\mathcal{N}_0(\upvarphi_{\h})(\y)-
    \|\Q^{\frac12}\mathcal{D}_{\y}\upvarphi_{\h}(\y)\|_{\H}^2
    \nonumber\\&=
    e^{2i(\y,\h)}\int_{Z}\big[e^{2i(\G(z),\h)}+1-2e^{i(\G(z),\h)}\big]\lambda(\d z)
    \nonumber\\&=
     \int_{Z}\big[\upvarphi_{\h}(\y+\G(z))-\upvarphi_{\h}(\y)\big]^2\lambda(\d z), \ 
     \y\in\H,
    \end{align*}
    which yields \eqref{carredujump}.
    Using the invariance of $\eta$ and integrating the identity \eqref{carredujump} over $\H$ with respect to $\eta$, we obtain
    \begin{align*}
    &	\int_{\H}\mathcal{N}_0\upvarphi(\y)\upvarphi(\y)\eta(\d \y) \nonumber\\&=-\frac{1}{2}\int_{\H}\|\sqrt{\Q}\mathcal{D}_{\y}\upvarphi(\y)\|_{\H}^2\eta(\d\y)
    -\frac12\int_{\H}
      \int_{Z}\big[\upvarphi_{\h}(\y+\G(z))-\upvarphi_{\h}(\y)\big]^2\lambda(\d z)
    	\eta(\d\y),
    \end{align*}
    which can equivalently be written as
     \begin{align}\label{carre}
    	&	\int_{\H}\mathcal{N}_0\upvarphi(\y)\upvarphi(\y)\eta(\d \y)  \nonumber\\&=-\frac{1}{2}\int_{\H}\|\sqrt{\Q}\mathcal{D}_{\y}\upvarphi(\y)\|_{\H}^2\eta(\d\y)
    	-\frac12\int_{\H}
    	\|\upvarphi_{\h}(\y+\G(z))-\upvarphi_{\h}(\y)\|_{\mathrm{L}^2_{\lambda}(Z)}^2
    	\eta(\d\y).
    \end{align}
    This identity shows that $\mathcal{N}_0$ is dissipative. As a consequence, since $\mathscr{E}_{\A}(\H)$ is dense in $\mathrm{L}^2(\H;\eta)$, the operator $\mathcal{N}_0$ is closable in $\mathrm{L}^2(\H;\eta)$.
    
     We now turn to the infinitesimal generator $\mathcal{N}_2$ associated with the semigroup $\{\mathrm{P}_t\}_{t\geq 0}$. Its domain $\D(\mathcal{N}_2)$ equipped with the graph norm
    \begin{align}\label{dn2d}
    	\|\upvarphi\|_{\D(\mathcal{N}_2)}^2=\|\upvarphi\|_{\mathrm{L}^2(\H;\eta)}^2
    	+\|\mathcal{N}_2\upvarphi\|_{\mathrm{L}^2(\H;\eta)}^2, \ \upvarphi\in\D(\mathcal{N}_2). 
    \end{align}
    Finally, since $\mathscr{E}_{\A}(\H)$ is a core of $\mathcal{N}_2$ (Theorem \ref{thm4.7}), the identity \eqref{carre} extends by density to all functions in $\D(\mathcal{N}_2)$, as stated below. 
    \begin{proposition}\label{carredu}
    	The operator $\Q^{\frac{1}{2}}\mathcal{D}_{\y}\upvarphi$, defined in $\mathscr{E}_{\A}(\H)$, admits a unique extension to a bounded linear operator from $\D(\mathcal{N}_2)$ into $\mathbb{L}^2(\H,\eta;\H)$. The extension is still denoted by
    	$\Q^{\frac{1}{2}}\mathcal{D}_{\y}\upvarphi$. Furthermore, the following ``Carre du Champ's" identity holds:
    	\begin{align}\label{453}
    		\int_{\H}\upvarphi\mathcal{N}_2\upvarphi\d\eta =-\frac{1}{2}\int_{\H}\|\sqrt{\Q}\mathcal{D}_{\y}\upvarphi\|_{\H}^2\d\eta
    		-\frac12\int_{\H}\|\upvarphi(\cdot+\G(\cdot))-\upvarphi(\cdot)\|_{\mathrm{L}^2_{\lambda}(Z)}^2\d\eta,
    	\end{align}
    	 for all $\upvarphi\in\D(\mathcal{N}_2)$. Moreover, we have following bounds:
    	\begin{align}\label{dn2d1}
    		\|\Q^{\frac{1}{2}}\mathcal{D}_{\y}\upvarphi\|_{\mathbb{L}^2(\H,\eta;\H)}\leq \|\upvarphi\|_{\D(\mathcal{N}_2)}
    		\text{ and } \|\upvarphi(\cdot+\G(\cdot))-\upvarphi(\cdot)\|_{\mathrm{L}^2(\H,\eta;
    		\mathrm{L}^2_{\lambda}(Z))}\leq\|\upvarphi\|_{\D(\mathcal{N}_2)},
    	\end{align}
    	for all $\upvarphi\in\D(\mathcal{N}_2)$.
    \end{proposition}
    \begin{proof}
    	From \eqref{carredujump} and \eqref{carre}, we have
       \begin{align}\label{carredujump1}
    \int_{\H}\upvarphi\mathcal{N}_0\upvarphi\d\eta &= -\frac{1}{2}\int_{\H}\|\sqrt{\Q}\mathcal{D}_{\y}\upvarphi\|_{\H}^2\d\eta
    -\frac12\int_{\H}\|\upvarphi(\cdot+\G(\cdot))-\upvarphi(\cdot)\|_{\mathrm{L}^2_{\lambda}(Z)}^2\d\eta,
    \end{align}
    for all $\upvarphi\in\mathscr{E}_{\A}(\H)$. Let $\upvarphi\in\D(\mathcal{N}_2)$. Then, there exists $\{\upvarphi_n\}_{n\in\N}$ in $\mathscr{E}_{\mathcal{A}}(\H)$ such that
    \begin{align}\label{ptconv}
    \upvarphi_n\to\upvarphi \ \text{ and } \ \mathcal{N}_0\upvarphi_n\to\mathcal{N}_2\upvarphi
    \ \text{ in } \ \mathrm{L}^2(\H;\eta) \ \text{ as } \ n\to\infty.
    \end{align}
    Let $m,n\in\N$ with $m<n$. Using \eqref{carredujump1}, we find 
    \begin{align*}
     &\int_{\H}\|\sqrt{\Q}\mathcal{D}_{\y}(\upvarphi_n-\upvarphi_m)\|_{\H}^2\d\eta+\int_{\H}
     \|\upvarphi_n(\cdot+\G(\cdot))-\upvarphi_n(\cdot)+\upvarphi_m(\cdot+\G(\cdot))-
     \upvarphi_m(\cdot)\|_{\mathrm{L}^2_{\lambda}(Z)}^2\d\eta
       \nonumber\\&\quad=-2
        \int_{\H}(\upvarphi_n-\upvarphi_m)(\mathcal{N}_0\upvarphi_n-\mathcal{N}_0\upvarphi_m)
        \d\eta.
    \end{align*}
    This shows that $\{\sqrt{\Q}\mathcal{D}_{\y}\upvarphi_n\}_{n\in\N}$ and 
    $\{\upvarphi_n(\cdot+\G(\cdot))-\upvarphi_n(\cdot)\}_{n\in\N}$ are Cauchy in 
    $\mathbb{L}^2(\H,\eta;\H)$ and $\mathrm{L}^2(\H,\eta;\mathrm{L}^2_{\lambda}(Z))$, respectively. Hence, by making the use of the strong convergence \eqref{ptconv}, the identity \eqref{453} follows immediately. Furthermore, by applying H\"older's inequality in \eqref{453} and incorporating \eqref{dn2d}, we finally deduce \eqref{dn2d1}.
    \end{proof}
    
    Note that we are not able to prove that the operator $\Q^{\frac{1}{2}}\D_{\x}\upvarphi$  defined in $\mathscr{E}_{\A}(\H)$
    is closable in $\mathrm{L}^2(\H,\eta)$.  
    \subsubsection{Resolvent estimates and perturbation results for $\mathcal{N}_2$}
    We next analyse resolvent and perturbation properties of the operator $\mathcal{N}_2$.
    \begin{proposition}[Estimate for resolvent of $\mathcal{N}_2$]\label{lem4.10}
    	Let $f\in\mathrm{L}^2(\H;\eta)$, $\kappa>0$ and set $\upvarphi=(\kappa\I-\mathcal{N}_2)^{-1}f$. Then we have 
    	\begin{align}\label{frd4}
    		\|\upvarphi\|_{\mathrm{L}^2(\H,\eta)}\leq \frac{1}{\sqrt{\kappa}}\|f\|_{\mathrm{L}^2(\H,\eta)}\ \text{ and }\  \|\Q^{\frac{1}{2}}\mathcal{D}_{\y}\upvarphi\|_{\mathbb{L}^2(\H,\eta;\H)}\leq \frac{1}{\sqrt{\kappa}}\|f\|_{\mathrm{L}^2(\H,\eta)}.
    	\end{align}
    \end{proposition}
    \begin{proof}
    	From the essential $m$-dissipativity of the Kolmogorov operator $\mathcal{N}_2$ (see Theorem \ref{thm4.7}), it follows that for every $\kappa>0$ and for every $f\in\mathrm{L}^2(\H;\eta)$, there exists a unique strong solution (in the sense of Friedrichs, see \eqref{frd1}-\eqref{frd2}) to the elliptic problem
    	\begin{align}\label{frd3}
    		\kappa\upvarphi-\mathcal{N}_2\upvarphi=f, \ \text{ in } \ \mathrm{L}^2(\H;\eta).
    	\end{align}
    	By multiplying equation \eqref{frd3} by $\upvarphi$, integrating with respect to $\eta$, and taking into account relation \eqref{453}, we obtain
    	\begin{align}\label{frd5}
    		\kappa\int_{\H}\upvarphi^2\d\eta&+
    		\frac{1}{2}\int_{\H}\|\sqrt{\Q}\D_{\x}\upvarphi\|_{\H}^2\d\eta
    		+\frac12\int_{\H}
    		\|\upvarphi(\cdot+\G(\cdot))-\upvarphi(\cdot)\|_{\mathrm{L}^2_{\lambda}(Z)}^2
    		\d\eta
    		=
    		\int_{\H}\upvarphi f\d\eta.
    	\end{align}
    	By applying Young's inequality in the right hand side of \eqref{frd5}, we deduce \eqref{frd4}.
    \end{proof}
The following result concerns a perturbation of the operator $\mathcal{N}_2$, which plays an important role in the application to optimal control problems involving the HJB equation.
    \begin{proposition}[Perturbation of $\mathcal{N}_2$]\label{pertubkol}
    	Let $\F\in\mathscr{B}_b(\H;\H)$ and $f\in\mathrm{L}^2(\H;\eta)$. Consider the linear operator 
    	\begin{align}
    		\mathcal{N}_1\upvarphi=\mathcal{N}_2\upvarphi+\big(\F(\y), \Q^{\frac{1}{2}}\mathcal{D}_{\y}\upvarphi\big),\ \upvarphi\in\D(\mathcal{N}_2). 
    	\end{align}	
    	Then, the resolvent set $\rho(\mathcal{N}_1)$ of $\mathcal{N}_1$ contains the half-right  $(\|\F\|_0^2,+\infty)$ and
    	\begin{align}\label{frd8}
    		(\kappa\I-\mathcal{N}_1)^{-1}f=(\kappa\I-\mathcal{N}_2)^{-1}(\I-
    		\mathcal{L}_{\kappa})^{-1}f, 
    	\end{align}
    	where 
    	\begin{align*}
    		\mathcal{L}_{\kappa}f(\y)=\big(\F(\y), \Q^{\frac{1}{2}}\mathcal{D}_{\y}(\kappa\I-\mathcal{N}_2)^{-1}f(\y)\big). 
    	\end{align*}
    \end{proposition}
    \begin{proof}
    	Consider the equation 
    	\begin{align}\label{frd6}
    		\kappa\upvarphi(\y)-\mathcal{N}_2\upvarphi(\y)-
    		\big(\F(\y), \Q^{\frac{1}{2}}\mathcal{D}_{\y}\upvarphi(\y)\big)=f(\y), \ \y\in\H.
    	\end{align}
    	Let us set $\kappa\upvarphi-\mathcal{N}_2\upvarphi=\uppsi$. Since, $\kappa\I-\mathcal{N}_2$ is invertible, equation \eqref{frd6} can be rewritten as 
    	\begin{align}\label{frd7}
    		\uppsi-\mathcal{L}_{\kappa}\uppsi=f,
    	\end{align}
    	where $\mathcal{L}_{\kappa}\uppsi=\big(\F(\y),\Q^{\frac{1}{2}}\mathcal{D}_{\y} (\kappa\I-\mathcal{N}_2)^{-1}\uppsi\big).$ In view of \eqref{frd4}, we find
    	\begin{align*}
    		\|\mathcal{L}_{\kappa} f\|_{\mathrm{L}^2(\H;\eta)}\leq
    		\|\F\|_{0} \|\Q^{\frac{1}{2}}\mathcal{D}_{\y}\upvarphi\|_{\mathbb{L}^2(\H,\eta;\H)}
    		\leq\frac{1}{\sqrt{\kappa}}\|\F\|_{0}\|f\|_{\mathrm{L}^2(\H,\eta)}.
    	\end{align*}
    	Therefore, whenever $\frac{1}{\sqrt{\kappa}}\|\F\|_{0}<1$, the Banach contraction mapping theorem ensures that equation \eqref{frd7} has a unique solution of the form
    	\begin{align*}
    		\uppsi=(\I-\mathcal{L}_{\kappa})^{-1}f,
    	\end{align*}
    	and the conclusion \eqref{frd8} follows immediately.
    \end{proof}
	\subsubsection{Some useful identities for the transition semigroup $\{\mathtt{P}_t\}_{t\geq 0}$}
	We next derive some useful identities for the transition semigroup $\{\mathtt{P}_t\}_{t\geq0}$ which will play a key role in solving the HJB equation in the next section.
	\begin{proposition}\label{ptp1}
		Let $\{\mathtt{P}_t\}_{t\geq0}$ be a transition semigroup in $\mathrm{L}^2(\H;\eta)$. For any $\upvarphi\in\D(\mathcal{N}_2)$ and any $t\geq0$, the following identity holds:
		\begin{align}\label{pt1}
		&\|\mathtt{P}_t\upvarphi\|_{\mathrm{L}^2(\H;\eta)}^2+\int_0^t
		\|\Q^{\frac12}\mathcal{D}_{\y}\mathtt{P}_{s}\upvarphi\|_{\mathbb{L}^2(\H,\eta;\H)}^2\d s
		+\int_0^t
		 \|\mathtt{P}_s\upvarphi(\cdot+\G(\cdot))-\mathtt{P}_s\upvarphi(\cdot)\|_{\mathrm{L}^2(\H,\eta;
			\mathrm{L}^2_{\lambda}(Z))}^2\d s
		\nonumber\\&=
		\|\upvarphi\|_{\mathrm{L}^2(\H;\eta)}^2.
		\end{align}
	\end{proposition}
	\begin{proof}
   By using the semigroup property and the fact that $\upvarphi\in\D(\mathcal{N}_2)$, one can show that $\mathtt{P}_t\upvarphi\in\D(\mathcal{N}_2)$. Furthermore, by the properties of semigroup (for instance, see \cite{EBD}), we find
   \begin{align*}
   	\frac{\d}{\d t}\|\mathtt{P}_t\upvarphi\|_{\mathrm{L}^2(\H;\eta)}^2&=
   	2\int_{\H}(\mathtt{P}_t\upvarphi)(\mathcal{N}_2\mathtt{P}_t\upvarphi)\d\eta
   	\nonumber\\&=
   	-\int_{\H}\|\sqrt{\Q}\mathcal{D}_{\y}\mathtt{P}_t\upvarphi\|_{\H}^2\d\eta
   	-\int_{\H}
   	\|\mathtt{P}_t\upvarphi(\cdot+\G(\cdot))-\mathtt{P}_t\upvarphi(\cdot)\|_{\mathrm{L}^2_{\lambda}(Z)}^2
   	\d\eta.
   \end{align*}
   On integrating above from $0$ to $t$, we deduce \eqref{pt1}.
	\end{proof}
	The following result is the direct consequence of Proposition \ref{ptp1}, together with the fact that $\D(\mathcal{N}_2)$ is dense in $\mathrm{L}^2(\H;\eta)$. Its proof follows the same arguments as those used in Proposition \ref{carredu}, and is therefore omitted.
	\begin{proposition}\label{ptp2}
	Let $\upvarphi\in\mathrm{L}^2(\H;\eta)$ and $t\geq0$. Then, for any $t_1>0$, the linear operator 
	\begin{align*}
		\upvarphi\in\D(\mathcal{N}_2)\subset\mathrm{L}^2(\H;\eta)\mapsto
		\Q^{\frac12}\mathcal{D}_{\y}\mathtt{P}_t\upvarphi\in\mathrm{L}^2(0,t_1;\mathbb{L}^2(\H,\eta;\H))
	\end{align*}
	can be uniquely extended to a linear operator from $\mathrm{L}^2(\H;\eta)$ to
	$\mathrm{L}^2(0,t_1;\mathbb{L}^2(\H,\eta;\H))$ (still denoted by the same symbol). Moreover, the following identity holds for any $t\geq0$:
	  \begin{align}\label{pt2}
		&\|\mathtt{P}_t\upvarphi\|_{\mathrm{L}^2(\H;\eta)}^2+\int_0^t
		\|\Q^{\frac12}\mathcal{D}_{\y}\P_{s}\upvarphi\|_{\mathbb{L}^2(\H,\eta;\H)}^2\d s
		+\int_0^t
		\|\mathtt{P}_s\upvarphi(\cdot+\G(\cdot))-\mathtt{P}_s\upvarphi(\cdot)\|_{\mathrm{L}^2(\H,\eta;
		\mathrm{L}^2_{\lambda}(Z))}^2\d s
		\nonumber\\&=
		\|\upvarphi\|_{\mathrm{L}^2(\H;\eta)}^2, \ \text{ for all } \ \upvarphi\in\mathrm{L}^2(\H;\eta).
	\end{align}
	\end{proposition}
%	\begin{proof}
%	Let $\{\upvarphi_n\}_{n\in\N}$ be a sequence in $\D(\mathcal{N}_2)$ and  $\upvarphi\in\mathrm{L}^2(\H;\eta)$ such that
%	\begin{align*}
%		\upvarphi_n\to\upvarphi \ \text{ in } \ \mathrm{L}^2(\H;\eta) \ \text{ and } \
%		\mathcal{N}_0\upvarphi_n\to\mathcal{N}_2\upvarphi_n \ \text{ in } \ \mathrm{L}^2(\H;\eta) \ \text{ as } \ n\to\infty.
%	\end{align*}
%	Proceeding in a similar way as we did in the Proposition \ref{carredu}, it follows that the sequences $\{\sqrt{\Q}\D_{\x}\mathtt{P}_s\upvarphi_n\}_{n\in\N}$ and 
%	$\{\mathtt{P}_s\upvarphi_n(\cdot+G(\cdot))-\mathtt{P}_s\upvarphi_n(\cdot)\}_{n\in\N}$ are Cauchy in 
%	$\mathrm{L}^2(0,t;\mathbb{L}^2(\H,\eta;\H))$ and $\mathrm{L}^2(0,t;\mathrm{L}^2(\H,\eta;\mathrm{L}^2_{\lambda}(Z)))$, respectively. Hence, the conclusion \eqref{pt2} follows.
%	\end{proof}
%	
	
	\section{Infinite horizon optimal control problem} \label{sec5}\setcounter{equation}{0} 
	
	This section is devoted to an application of the Kolmogorov equation associated with the SCBF system \eqref{scbfabs} to optimal control problems, with particular  emphasis on the infinite-horizon case. Further discussions and analysis of such optimal control problems can be found in  \cite[Chapters 4 and 5]{GFAS}. For broader perspectives and detailed developments in stochastic control of infinite-dimensional systems, including evolution equations and SPDEs, we refer the reader to monograph \cite{QLXZ} (also see \cite{QL1} for infinite horizon problem).
	For a given  time $t>0$, we consider the following abstract controlled SCBF equations driven with Gaussian and L\'evy noise:
		\begin{equation}\label{scbfabscon}
		\left\{
		\begin{aligned}
			\d\Y(t)&+[\mu\mathcal{A}\Y(t)+\mathfrak{B}(\Y(t))+\alpha\Y(t) +\beta\mathfrak{C}(\Y(t))+\sqrt{\Q}\mathrm{U}(t)]\d t \\&=\sqrt{\Q}\d\W(t)+\int_{Z}\G(z)\wi{\pi}
			(\d t,\d z), \ t>0,\\
			\Y(0)&=\y\in\H,
		\end{aligned}
		\right.
	\end{equation}
	where $\mathrm{U}(\cdot)$ is some control process taking values in $\B(0,\mathpzc{R}):=\left\{\boldsymbol{x}\in\H:\|\boldsymbol{x}\leq\mathpzc{R}
	\right\}\subset\H$, for $\mathpzc{R}>0$. We denote by  $$\mathrm{L}^{2,\kappa}(\Omega;\mathrm{L}^{2}(0,+\infty;\H)),$$
	the space of all progressively measurable processes 
	$\mathpzc{V}:(0,\infty)\times\Omega\to\H$ 
	such that 
	\begin{align*}
		\E\left[\int_0^{\infty}e^{-\kappa s}\|\mathpzc{V}(s)\|_{\H}^2\d s\right]<\infty,
	\end{align*}
	where $\kappa>0$ is the discount factor. We define the admissible class of control processes as
	\begin{align}\label{admis}
		\mathcal{U}_{\mathrm{ad}}:=\{\mathrm{U}(\cdot)\in
		\mathrm{L}^{2,\kappa}(\Omega;\mathrm{L}^{2}(0,+\infty;\H)):
		\|\mathrm{U}(t)\|_{\H}\leq\mathpzc{R}, \ \mathbb{P}\text{-a.s.} \},
	\end{align}
	corresponding to a fixed reference probability space $(\Omega,\mathscr{F},\mathbb{P})$. According to Theorem \ref{extunjump}, for any $\x\in\H$ and for any control 
	$\mathpzc{U}(\cdot)\in\mathcal{U}_{\mathrm{ad}}$, the  system \eqref{scbfabscon} has a unique solution in the sense of Definition \ref{defjumpe}-\ref{defjumpu}.
	In connection with the system \eqref{scbfabscon}, we study the following stochastic optimal control problem of minimizing the discounted cost functional (for the fixed discount factor $\kappa>0$) of the form
	\begin{align}\label{cost}
		\mathcal{J}_{\infty}(\y,\mathrm{U})=\E\left\{\int_0^{\infty}e^{-\kappa s} \left[f(\Y(s,\y;\mathrm{U}(\cdot)))+h(\mathrm{U}(s))\right]\d s\right\},
	\end{align}
	over all controls $\mathrm{U}(\cdot)\in\mathcal{U}_{\mathrm{ad}}$. Here $\Y(\cdot)$ is the unique solution of \eqref{scbfabscon}, $f:\H\to\R$ is some bounded measurable function, and $h:\H\to(-\infty,+\infty]$ is a convex and lower semicontinuous function. Corresponding to \eqref{admis}, we define the value function
	\begin{align}\label{56}
		\mathfrak{V}(\y):=
		\inf_{\mathrm{U}(\cdot)\in\mathcal{U}_{\mathrm{ad}}}
		\mathcal{J}_{\infty}(\y,\mathpzc{U})=
		\inf_{\mathrm{U}(\cdot)\in\mathcal{U}_{\mathrm{ad}}}\E\left\{\int_0^{\infty}
		e^{-\kappa s}\left[f(\Y(s,\y;\mathrm{U}(\cdot)))+h(\mathrm{U}(s))\right]\d s \right\}.
	\end{align}
%The value function $\mathfrak{V}$ satisfies the following dynamic programming principle whose proof is given in the Appendix \ref{dppinfinite}.
%\begin{proposition}
%	For every $t>0$, $\x\in\H$, the value function \eqref{56} satisfies the following:
%	\begin{align*}
%		\mathfrak{V}(\x)=\inf_{\mathpzc{U}(\cdot)\in\mathcal{U}_{\mathrm{ad}}}
%		\E\left\{\int_0^{t}
%		e^{-\lambda s}\left[f(\X(s,\x;\mathpzc{U}(\cdot)))+h(\mathpzc{U}(s))\right]
%		\d s+
%		e^{-\lambda t}\mathfrak{V}(\X(t)) \right\}.
%	\end{align*}
%\end{proposition}
Adopting the dynamic programming approach, we seek a solution (whose precise definition is given below) of the following  infinite-dimensional second order stationary HJB equation associated with the stochastic optimal control problem \eqref{scbfabscon}-\eqref{cost}:
\begin{align}\label{hjbn0}
	&\kappa\upvarphi(\y)-\frac{1}{2}\Tr\left[\Q\mathcal{D}_{\y}^2\upvarphi(\y)\right]+ (\mu\A\y+\alpha\y+\B(\y)+\beta\mathfrak{C}(\y),\mathcal{D}_{\y}\upvarphi(\y))
	\nonumber\\&\quad-
	\int_{Z}\left[\uppsi(\y+\G(z))-\uppsi(\y)-(\G(z),\mathcal{D}_{\y}\uppsi(\y))\right]\lambda(\d z) +g(\Q^{1/2}\mathcal{D}_{\y}\upvarphi(\y))=f(\y), \ \text{ for } \ \y\in\H,
\end{align}
where the Hamiltonian $g:\H\to\R$ is the Legendre transform of $h$, which is defined by
\begin{align}\label{leg}
	g(\y)=\sup\limits_{\x\in\H}\left\{(\y,\x)-h(\x)\right\},\ \y\in\H.
\end{align}

\subsection{Kolmogorov formulation for the HJB equation}
In this section, we apply the Kolmogorov framework developed in Sections \ref{secvarmea}-\ref{disscore} to analyse the stationary HJB equation \eqref{hjbn0} associated with the infinite-horizon optimal control problem \eqref{scbfabscon}-\eqref{cost}. The objective is to characterize the value function  through an appropriate solution concept compatible with the HJB equation \eqref{hjbn0}. We write the equation \eqref{hjbn0} in the folllwing abstract form 
\begin{align}\label{551}
\kappa\upvarphi-\mathcal{N}_0\upvarphi+g(\Q^{1/2}\mathcal{D}_{\y}\upvarphi)=f,
\end{align}
where $\mathcal{N}_0$ is the Kolmogorov differential operator defined in \eqref{4p5} and the nonlinearity $g$ coincides with the Hamiltonian. Therefore, solving the Kolmogorov equation \eqref{551} is equivalent to solving the HJB equation \eqref{hjbn0}.  

Note that the equation \eqref{551} should be viewed only as formal identity. Indeed, $\mathcal{N}_0$ is not closed in $\mathrm{L}^2(\H,\eta)$ and thus \eqref{551} is not defined as a well-posed problem in $\mathrm{L}^2(\H,\eta)$. The correct framework arises from the fact that $\mathcal{N}_0$ is closable in $\mathrm{L}^2(\H,\eta)$ and its closure coincides with $\mathcal{N}_2$, the infinitesimal generator of the transition semigroup $\{\mathtt{P}_t\}_{t\geq 0}$ (see Theorem \ref{thm4.7}). Therefore, the HJB
equation \eqref{551} must be understood in the following sense:
\begin{align}\label{555}
\kappa\upvarphi-\mathcal{N}_2\upvarphi+g(\Q^{1/2}\mathcal{D}_{\y}\upvarphi)=f \ \text{ in } \ 
\mathrm{L}^2(\H,\eta).
\end{align}
In the above form, the equation is meaningful, because $\mathcal{N}_2$ is a closed operator in 
$\mathrm{L}^2(\H,\eta)$ and its resolvent $(\kappa\I-\mathcal{N}_2)^{-1}:\mathrm{L}^2(\H,\eta)\to\D(\mathcal{N}_2)
\subset\mathrm{L}^2(\H,\eta)$ is well-defined. Consequently, one can apply the resolvent estimates (see Proposition \ref{lem4.10}) and the Carre du Champ bounds (see Proposition \ref{carredu}) within the natural Hilbert space $\mathrm{L}^2(\H,\eta)$. By using the associated resolvent operator $(\kappa\I-\mathcal{N}_2)^{-1}:\mathrm{L}^2(\H,\eta)\to\D(\mathcal{N}_2)
\subset\mathrm{L}^2(\H,\eta)$, defined by
\begin{align*}
(\kappa\I-\mathcal{N}_2)^{-1}\phi=
\int_0^{\infty} e^{-\kappa t} \mathtt{P}_t \phi \,\d t, \ \phi\in \mathrm{L}^2(\H,\eta),
\end{align*}
the equation \eqref{555} can be written in the following form: 
\begin{align}\label{559}
\upvarphi=(\kappa\I-\mathcal{N}_2)^{-1}(f-g(\Q^{1/2}\mathcal{D}_{\y}\upvarphi))
=\int_0^{\infty} e^{-\kappa t} \mathtt{P}_t\big[f-g(\Q^{\frac12}\mathcal{D}_{\y}\upvarphi)\big](\x)\,\d t.
\end{align}
		
\begin{definition}[Mild solution]\label{genmild}
Let $f\in\mathrm{L}^2(\H,\eta)$ and $g\in\mathrm{Lip}_b(\H)$. We say that $\upvarphi\in\mathrm{L}^2(\H,\eta)$ such that $\Q^{\frac12}\mathcal{D}_{\y}\upvarphi\in\L^2(\H,\eta;\H)$ is a mild solution of \eqref{551} if and only if $\upvarphi$ satisfies the following integral equation for all $\x\in\H$:
\begin{align}\label{fxdp}
\upvarphi(\x)=\int_0^{\infty} e^{-\kappa t} \mathtt{P}_t\big[f-g(\Q^{\frac12}\mathcal{D}_{\y}\upvarphi)\big](\x)\,\d t.
\end{align}
\end{definition}
\subsubsection{Solvability of the HJB equation} 
Having identified the appropriate notion of solution for the HJB equation \eqref{hjbn0}, we now address its solvability. As discussed above, the strict formulation is too restrictive in the present infinite-dimensional setting due to the lack of the sufficient regularity of the value function. We therefore focus on the existence of  mild solutions introduced above.

\begin{theorem}\label{HJBeqn1}
Assume that the Hamiltonian $g:\H\to\R$ is Lipschitz continuous and $f\in\mathrm{L}^2(\H;\eta)$. Then, for 
$\frac{1}{\kappa}+\frac{1}{\sqrt{\kappa}}<\frac{1}{\|g\|_{\mathrm{Lip}}}$, the HJB equation \eqref{hjbn0} has a unique mild solution $\upvarphi$ in the sense of Definition \ref{genmild}.
\end{theorem}

\begin{proof}
We aim to now show that the HJB equation \eqref{hjbn0} has a unique mild solution $\upvarphi\in\mathrm{L}^2(\H,\eta)$ which is given by the resolvent formula
\begin{align*}
\upvarphi=(\kappa\I-\mathcal{N}_2)^{-1}(f-g(\Q^{\frac12}\mathcal{D}_{\y}\upvarphi)).
\end{align*}
Let us first define the space 
\begin{align*}
\mathbb{X}:=\mathrm{L}^2(\H,\eta)\times\L^2(\H,\eta;\H).
\end{align*}
We equip the space $\mathbb{X}$ with respect to the following norm:
\begin{align*}
\|(\upvarphi,\boldsymbol{\uppsi})\|_{\mathbb{X}}:=\|\upvarphi\|_{\mathrm{L}^2(\H,\eta)}+\|\boldsymbol{\uppsi}\|_{\L^2(\H,\eta;\H)},
\end{align*}
and the space $\mathbb{X}$ is complete with respect to the above norm. We now introduce a map which allows us to use a fixed point argument. Define the map $\mathfrak{G}:\mathbb{X}\to\mathbb{X}$ by
\begin{align*}
\mathfrak{G}(\upvarphi,\boldsymbol{\uppsi}):=\big((\kappa\I-\mathcal{N}_2)^{-1}(g(\boldsymbol{\uppsi})+f),
\Q^{\frac12}\mathcal{D}_{\y}(\kappa\I-\mathcal{N}_2)^{-1}(g(\boldsymbol{\uppsi})+f)\big),
\end{align*}
where $g:\H\to\R$ is a Lipschitz continuous function and $f\in\mathrm{L}^2(\H,\eta)$. 
\vskip 0.2cm
\noindent
\emph{The map $\mathfrak{G}$ is well-defined:} Let $(\upvarphi,\boldsymbol{\uppsi})\in\mathbb{X}$. Then, since $\boldsymbol{\uppsi}\in\L^2(\H,\eta;\H)$ and $g$ is Lipschitz continuos, we immediately have $g(\boldsymbol{\uppsi})\in\mathrm{L}^2(\H,\eta)$.
Moreover, from \eqref{frd4} and the Lipschitz continuity of $g$ yields:
\begin{align}\label{Gwedef}
\|\mathfrak{G}(\upvarphi,\boldsymbol{\uppsi})\|_{\mathbb{X}}&=
\|\big((\kappa\I-\mathcal{N}_2)^{-1}(g(\boldsymbol{\uppsi})+f),
\Q^{\frac12}\mathcal{D}_{\y}(\kappa\I-\mathcal{N}_2)^{-1}(g(\boldsymbol{\uppsi})+f)\big)\|_{\mathbb{X}}
\nonumber\\&=
\|(\kappa\I-\mathcal{N}_2)^{-1}(g(\boldsymbol{\uppsi})+f)\|_{\mathrm{L}^2(\H,\eta)}
+\|\Q^{\frac12}\mathcal{D}_{\y}(\kappa\I-\mathcal{N}_2)^{-1}(g(\boldsymbol{\uppsi})+f)\|_{\L^2(\H,\eta;\H)}
\nonumber\\&\leq
\|(\kappa\I-\mathcal{N}_2)^{-1}g(\boldsymbol{\uppsi})\|_{\mathrm{L}^2(\H,\eta)}+
\|(\kappa\I-\mathcal{N}_2)^{-1}f\|_{\mathrm{L}^2(\H,\eta)}
\nonumber\\&\quad+
\|\Q^{\frac12}\mathcal{D}_{\y}(\kappa\I-\mathcal{N}_2)^{-1}g(\boldsymbol{\uppsi})\|_{\L^2(\H,\eta;\H)}+
\|\Q^{\frac12}\mathcal{D}_{\y}(\kappa\I-\mathcal{N}_2)^{-1}f\|_{\L^2(\H,\eta;\H)}
\nonumber\\&\leq
\frac{1}{\kappa}\|g(\boldsymbol{\uppsi})\|_{\mathrm{L}^2(\H,\eta)}+
\frac{1}{\kappa}\|f\|_{\mathrm{L}^2(\H,\eta)}
+
\frac{1}{\sqrt{\kappa}}\|g(\boldsymbol{\uppsi})\|_{\mathrm{L}^2(\H,\eta)}+
\frac{1}{\sqrt{\kappa}}\|f\|_{\mathrm{L}^2(\H,\eta)},
\end{align}
which in turn implies that $\mathfrak{G}(\upvarphi,\boldsymbol{\uppsi})\in\mathbb{X}$ and thus, the map 
$\mathfrak{G}:\mathbb{X}\to\mathbb{X}$ is well-defined.
\vskip 0.2cm
\noindent
\emph{The map $\mathfrak{G}$ is a contraction.} Let $(\upvarphi_1,\boldsymbol{\uppsi}_1),(\upvarphi_2,\boldsymbol{\uppsi}_2)\in\mathbb{X}$. Similar to \eqref{Gwedef}, we compute
\begin{align*}
&\|\mathfrak{G}(\upvarphi_1,\boldsymbol{\uppsi}_1)-\mathfrak{G}(\upvarphi_2,\boldsymbol{\uppsi}_2)\|_{\mathbb{X}}
\nonumber\\&\leq
\|(\kappa\I-\mathcal{N}_2)^{-1}\big((g(\boldsymbol{\uppsi}_1)+f)-(g(\boldsymbol{\uppsi}_2)+f)
\big)\|_{\mathrm{L}^2(\H,\eta)}
\nonumber\\&\quad+
\|\Q^{\frac12}\mathcal{D}_{\y}(\kappa\I-\mathcal{N}_2)^{-1}
\big((g(\boldsymbol{\uppsi}_1)+f)-(g(\boldsymbol{\uppsi}_2)+f)\big)\|_{\L^2(\H,\eta;\H)}
\nonumber\\&\leq
\|(\kappa\I-\mathcal{N}_2)^{-1}
\underbrace{
(g(\boldsymbol{\uppsi}_1)-g(\boldsymbol{\uppsi}_2))}_{\in\mathrm{L}^2(\H,\eta)}\|_{\mathrm{L}^2(\H,\eta)}
+
\|\Q^{\frac12}\mathcal{D}_{\y}(\kappa\I-\mathcal{N}_2)^{-1}\underbrace{
(g(\boldsymbol{\uppsi}_1)-g(\boldsymbol{\uppsi}_2))}_{\in\mathrm{L}^2(\H,\eta)}\|_{\L^2(\H,\eta;\H)}
\nonumber\\&\leq
\frac{1}{\kappa}\|g(\boldsymbol{\uppsi}_1)-g(\boldsymbol{\uppsi}_2)\|_{\mathrm{L}^2(\H,\eta)}+
\frac{1}{\sqrt{\kappa}}\|g(\boldsymbol{\uppsi}_1)-g(\boldsymbol{\uppsi}_2)\|_{\mathrm{L}^2(\H,\eta)}
\nonumber\\&\leq
\|g\|_{\mathrm{Lip}}\left(\frac{1}{\kappa}+\frac{1}{\sqrt{\kappa}}\right)
\|\boldsymbol{\uppsi}_1-\boldsymbol{\uppsi}_2\|_{\L^2(\H,\eta;\H)}
\nonumber\\&\leq
\|g\|_{\mathrm{Lip}}\left(\frac{1}{\kappa}+\frac{1}{\sqrt{\kappa}}\right)
\|(\upvarphi_1,\boldsymbol{\uppsi}_1)-(\upvarphi_2,\boldsymbol{\uppsi}_2)\|_{\mathbb{X}}.
\end{align*}
Thus, for $\|g\|_{\mathrm{Lip}}\left(\frac{1}{\kappa}+\frac{1}{\sqrt{\kappa}}\right)<1$, the map $\mathscr{G}$ is a contraction map. Thus, the Banach fixed point theorem gives a unique fixed point $(\upvarphi,\boldsymbol{\uppsi})\in\mathbb{X}$ such that $\mathscr{G}(\upvarphi,\boldsymbol{\uppsi})=(\upvarphi,\boldsymbol{\uppsi})$. Consequently, we have 
\begin{align*}
\upvarphi=(\kappa\I-\mathcal{N}_2)^{-1}(g(\boldsymbol{\uppsi})+f) \ \text{ and } \ 
\boldsymbol{\uppsi}=\Q^{\frac12}\mathcal{D}_{\y}(\kappa\I-\mathcal{N}_2)^{-1}(g(\boldsymbol{\uppsi})+f).
\end{align*}
Thus, we derive 
\begin{align*}
\upvarphi=(\kappa\I-\mathcal{N}_2)^{-1}(g(\Q^{\frac12}\mathcal{D}_{\y}\upvarphi)+f),
\end{align*}
which completes the proof.
\end{proof}

\begin{appendix}\renewcommand{\thesection}{\Alph{section}}
	\numberwithin{equation}{section}
	\section{Existence of an optimal control}\label{existopt}
For the existence of an optimal control, we take our hamiltonian as given below. Consider the function (see \cite[pp. 294]{gdp7})
\begin{align}\label{epe}
h(\x)=
\left\{\begin{array}{cc}\frac{1}{2}\|\x\|_{\H}^2, &\text{ if } \|\x\|_{\H}\leq \mathpzc{R},\\
+\infty,  &\text{ if } \|\x\|_{\H}>\mathpzc{R}.
\end{array}\right.
\end{align}
Then, the hamiltonian $g(\cdot)$ is explicitly given by 
\begin{align}\label{how1}
g(\x)=
\left\{\begin{array}{cc}\frac{1}{2}\|\x\|_{\H}^2, &\text{ if } \|\x\|_{\H}\leq \mathpzc{R},\\
\mathpzc{R}\|\x\|_{\H}-\frac{\mathpzc{R}^2}{2}, &\text{ if } \|\x\|_{\H}>\mathpzc{R}.
\end{array}
\right.
\end{align}
%		Moreover, the optimal feedback control is given formally as  (see \cite{gozzi})
%		\begin{align*}
%			\wi{\mathrm{U}}(t)=-\mathscr{G}(\Q^{1/2}\mathcal{D}_{\y}\upvarphi(\x)),
%		\end{align*}
%		where $\mathscr{G}(\cdot)$ is given by
%		\begin{align}\label{how2}
%			\mathscr{G}(\x):=
%			\begin{cases}
%				\x,  &\text{ when } \|\x\|_{\H}\leq R,\\
%				\frac{R}{\|\x\|_{\H}}\x, &\text{ when } \|\x\|_{\H}>R.
%			\end{cases}
%		\end{align} 

%	Let us fix the Hamiltonian  (see example \eqref{epe})
%	\begin{align*}
%		g(\x)=\left\{\begin{array}{cc}\frac{1}{2}\|\x\|_{\H}^2, &\text{ if } \ \|\x\|_{\H}\leq R,\\
%			R\|\x\|_{\H}-\frac{R^2}{2}, &\text{ if }\  \|\x\|_{\H}>R.\end{array}\right.
%	\end{align*}
We need the following lemma in the sequel.
\begin{lemma}\label{verif}
Let $\mathrm{U}(\cdot)\in\mathcal{U}_{\mathrm{ad}}$ and $\upvarphi(\cdot)$ be the mild solution to \eqref{hjbn0}. Then, the following identity holds:
\small{ \begin{align}\label{713}
& \upvarphi(\x)+\E\left(\int_0^{\infty}\frac{e^{-\kappa t}}{2} \left[\|\mathrm{U}(t)+\Q^{1/2}\mathcal{D}_{\y}\upvarphi(\Y(t,\x))\|_{\H}^2 -\uppsi(\|\Q^{1/2}\mathcal{D}_{\y}\upvarphi(\Y(t,\x))\|_{\H}-\mathpzc{R})\right] \d t\right)\nonumber\\&\quad=\mathcal{J}_{\infty}(\mathrm{U}),
\end{align}}
where 
\begin{align*}
\uppsi(\upxi):=
\begin{cases}
\upxi^2, \ &\text{ if } \ \upxi>0,\\
0, \ &\text{ if } \ \upxi\leq0.
\end{cases}
\end{align*}
\end{lemma}
\begin{proof}
Without loss of generality, we assume that $\upvarphi\in\C_b^2(\H)$ (otherwise one can proceed via Galerkin approximations, \cite[Chapter 13]{gdp7}). Since, $\Y(\cdot)$ is the unique strong solution of \eqref{scbfabscon},on applying the It\^o formula to the $t\mapsto e^{-\kappa t}\upvarphi(\Y(t,\x))$ yields, $\mathbb{P}-$a.s. 
\begin{align}\label{stpr}
&e^{-\kappa t}\upvarphi(\Y(t,\x))
\nonumber\\&=\upvarphi(\x)-\kappa
\int_0^t e^{-\kappa s}\upvarphi(\Y(s,\x))\d s
\nonumber\\&\quad-
\int_0^t e^{-\kappa s}
\big(\mu\mathcal{A}\Y(s)+\alpha\Y(s)+\mathfrak{B}(\Y(s))+\beta\mathfrak{C}(\Y(s)),
\mathcal{D}_{\y}\upvarphi(\Y(s))\big)\d s
\nonumber\\&\quad+
\int_0^t e^{-\kappa s}\big(\sqrt{\Q}\mathrm{U}(s),\mathcal{D}_{\y}\upvarphi(\Y(s))\big)+
\frac{1}{2} \int_0^t e^{-\kappa s}\Tr(\Q\D^2_{\x}\upvarphi(\Y(s)))\d s
\nonumber\\&\quad+
\int_0^t e^{-\kappa s}\big(\sqrt{\Q}\d\W(s),\mathcal{D}_{\y}\upvarphi(\Y(s))\big)+
\int_0^t\int_{Z}e^{-\kappa s}\big[\upvarphi(\Y(s-)+\G(z))-\upvarphi(\Y(s-))\big]
\wi{\pi}(\d s,\d z)
\nonumber\\&\quad+
\int_0^t\int_{Z}e^{-\kappa s}\big[\upvarphi(\Y(s-)+\G(z))-\upvarphi(\Y(s-))-
(\mathcal{D}_{\y}\upvarphi(\Y(s-)),\G(z))\big]\lambda(\d z)\d s,
\end{align}
for all $t\in[0,T]$. 
By utilizing \eqref{hjbn0} in \eqref{stpr}, it reduces to $\mathbb{P}-$a.s.
\begin{align}\label{gvaca}
&e^{-\kappa t}\upvarphi(\Y(t,\x))-\upvarphi(\x)
\nonumber\\&=
\int_0^t e^{-\kappa s}\big(\mathrm{U}(s),\sqrt{\Q}\mathcal{D}_{\y}\upvarphi(\Y(s))\big)\d s+
\underbrace{
\int_0^t e^{-\kappa s}\big(\sqrt{\Q}\d\W(s), \mathcal{D}_{\y}\upvarphi(\Y(s))\big)}_{\mathscr{F}_t-\text{ adapted martingale }}
\nonumber\\&\quad+
\underbrace{
	\int_0^t\int_{Z}e^{-\kappa s}\big[\upvarphi(\Y(s-)+\G(z))-\upvarphi(\Y(s-))\big]
	\wi{\pi}(\d s,\d z)}_{\mathscr{F}_t-\text{ adapted martingale }}
\nonumber\\&\quad+
\int_0^t e^{-\kappa s}g(\Q^{1/2}\mathcal{D}_{\y}\upvarphi(\Y(s,\x)))\d s
-\int_0^t e^{-\kappa s} f(\Y(s,\x))\d s,
\end{align} 
for all $t\in[0,T]$. On taking the expectation, and then using the fact that martingale terms in \eqref{gvaca} having zero expectation, we arrive at
\begin{align}\label{gvaca1}
&\E\big[e^{-\kappa t}\upvarphi(\Y(t,\x))\big]
\nonumber\\&=\upvarphi(\x)+\E\bigg[
\int_0^t e^{-\kappa s}\big(\mathrm{U}(s),\sqrt{\Q}\mathcal{D}_{\y}\upvarphi(\Y(s))\big)\d s \bigg]
\nonumber\\&\quad+
\E\bigg[\int_0^t e^{-\kappa s}g(\Q^{1/2}\mathcal{D}_{\y}\upvarphi(\Y(s,\x)))\d s\bigg]
-\E\bigg[\int_0^t e^{-\kappa s} f(\Y(s,\x))\d s\bigg],
\end{align} 
for all $t\in[0,T]$. We now consider the following two cases:
\vskip 0.2cm
\noindent
\textbf{Case-1:} \emph{When $\|\Q^{1/2}\mathcal{D}_{\y}\upvarphi(\Y(\cdot))\|_{\H}\leq \mathpzc{R}$.} In this case, we have 
\begin{align}\label{gvaca3}
g(\Q^{1/2}\mathcal{D}_{\y}\upvarphi(\Y(\cdot)))
&=\frac12\|\Q^{1/2}\mathcal{D}_{\y}\upvarphi(\Y(\cdot))\|_{\H}^2
\nonumber\\&=
\frac12\|\mathrm{U}(\cdot)+\Q^{1/2}\mathcal{D}_{\y}\upvarphi(\Y(\cdot))\|_{\H}^2-
\frac12\|\mathrm{U}(\cdot)\|_{\H}^2-(\mathrm{U}(\cdot),
\Q^{1/2}\mathcal{D}_{\y}\upvarphi(\Y(\cdot))).
\end{align}
On substituting \eqref{gvaca3} into \eqref{gvaca1}, we obtain
\begin{align}\label{gvaca2}
&\E\big[e^{-\kappa t}\upvarphi(\Y(t,\x))\big]
\nonumber\\&=\upvarphi(\x)
-\frac12\E\bigg[
\int_0^t e^{-\kappa s}\|\mathrm{U}(s)\|_{\H}^2\d s\bigg]
\nonumber\\&\quad+
\frac12\E\bigg[\int_0^t e^{-\kappa s}
\|\mathrm{U}(s)+\Q^{1/2}\mathcal{D}_{\y}\upvarphi(\Y(s,\x))\|_{\H}^2\d s\bigg]
-\E\bigg[\int_0^t e^{-\kappa s} f(\Y(s,\x))\d s\bigg],
\end{align} 
for all $t\in[0,T]$. On taking the limit supremum as $t\to\infty$ and utilizing \eqref{cost}, the above equality reduces to 
\begin{align}\label{cost1}
\upvarphi(\x)&=\E\bigg( \int_0^{\infty} e^{-\kappa t}\left[ f(\Y(t,\x))+\frac{1}{2} \|\mathrm{U}(t)\|_{\H}^2\right]\d t-\int_0^{\infty} \frac{e^{-\kappa t}}{2}\|\mathrm{U}(t)+\Q^{1/2} \mathcal{D}_{\y}\upvarphi(\Y(t,\x))\|_{\H}^2\d t \bigg)\no\\&=\mathcal{J}_{\infty}(\mathrm{U})-
\E\left(\int_0^{\infty}\frac{e^{-\kappa t}}{2} \|\mathrm{U}(t)+\Q^{1/2}\mathcal{D}_{\y}\upvarphi(\Y(t,\x))\|_{\H}^2 \d t\right).
\end{align}
\vskip 0.2cm
\noindent
\textbf{Case-2:} \emph{When $\|\Q^{1/2}\mathcal{D}_{\y}\upvarphi(\Y(\cdot))\|_{\H}> \mathpzc{R}$.} 
In this case, we have 
\begin{align}\label{gvaca4}
g(\Q^{1/2}\mathcal{D}_{\y}\upvarphi(\Y(\cdot)))
&=\mathpzc{R}\|\Q^{1/2}\mathcal{D}_{\y}\upvarphi(\Y(\cdot))\|_{\H}-\frac{\mathpzc{R}^2}{2}
\nonumber\\&=
-\frac12(\|\Q^{1/2}\mathcal{D}_{\y}\upvarphi(\Y(\cdot))\|_{\H}-\mathpzc{R})^2+
\frac12\|\Q^{1/2}\mathcal{D}_{\y}\upvarphi(\Y(\cdot))\|_{\H}^2.
\end{align}
On substituting \eqref{gvaca4} into \eqref{gvaca1}, we obtain
\begin{align}\label{gvaca5}
&\E\big[e^{-\kappa t}\upvarphi(\Y(t,\x))\big]
\nonumber\\&=\upvarphi(\x)+
\frac12\E\bigg[\int_0^t e^{-\kappa s}
\|\mathrm{U}(s)+\Q^{1/2}\mathcal{D}_{\y}\upvarphi(\Y(s,\x))\|_{\H}^2\d s\bigg]
-\frac12\E\bigg[\int_0^t e^{-\kappa s}\|\mathrm{U}(s)\|_{\H}^2\d s\bigg]
\nonumber\\&\quad-
\frac12\E\bigg[\int_0^t e^{-\kappa s}
(\|\Q^{1/2}\mathcal{D}_{\y}\upvarphi(\Y(s,\x))\|_{\H}-\mathpzc{R})^2\d s\bigg]
-\E\bigg[\int_0^t e^{-\kappa s} f(\Y(s,\x))\d s\bigg],
\end{align} 
for all $t\in[0,T]$. Consequently, the argument as before, we arrive at the representation
\begin{align}\label{cost2}
\upvarphi(\y)=\mathcal{J}_{\infty}(\mathrm{U})-\E\left(\int_0^{\infty}\frac{e^{-\kappa t}}{2}\left[\|\mathrm{U}(t)+ \Q^{1/2}\mathcal{D}_{\y}\upvarphi(\Y(t,\y))\|_{\H}^2 -(\|\Q^{1/2}\mathcal{D}_{\y}\upvarphi(\Y(t,\y))\|_{\H}-\mathpzc{R})^2\right] \d t\right).
\end{align}
Combining \eqref{cost1} and \eqref{cost2}, we obtain \eqref{713}, which completes the proof.
\end{proof}  
The existence of an optimal control for problem \eqref{56} follows from Lemma \ref{verif}, using arguments analogous to those in \cite[Section 13.4.2]{gdp7} (see \cite{gozzi1,gozzi2} also). This summarized in the following result:
\begin{theorem}
There exists an optimal pair $(\mathrm{U}^*(\cdot),\Y^*(\cdot))$ for problem \eqref{56} such that the optimal feedback law is given by
\begin{equation*}
\mathrm{U}^*(t)=\mathcal{G}(\Q^{1/2}\mathcal{D}_{\y}\upvarphi(\Y^*(t,\y))), \ \text{ for all } \ t\geq 0 \ \text{ and } \ \y\in\H,
\end{equation*}
where the map $\mathcal{G}:\H\to\H$ is defined by
\begin{equation*}
\mathcal{G}(\mathfrak{p})=\mathcal{D}_{\mathfrak{p}} g(\mathfrak{p})=
\left\{
\begin{aligned}
-\mathfrak{p}, \ \text{ when } \|\mathfrak{p}\|_{\H}\leq \mathpzc{R},\\
-\mathpzc{R}\frac{\mathfrak{p}}{\|\mathfrak{p}\|_{\H}}, \ \text{ when } \|\mathfrak{p}\|_{\H}> 
\mathpzc{R}.
\end{aligned}
\right.
\end{equation*}
Moreover, the optimal cost is characterized by $\mathcal{J}^*_{\infty}(\y)=\upvarphi(\y)$.
\end{theorem}

\end{appendix}

	\medskip
\noindent
\textbf{Acknowledgments:} The first author gratefully acknowledges financial support from Ministry of Education, Government of India - MHRD for financial assistance. M. T. Mohan's research was supported by the National Board of Higher Mathematics (NBHM), Department of Atomic Energy, Government of India (Project No. 02011/13/2025/NBHM(R.P)/R\&D II/1137).

\end{document}